\documentclass[11pt]{preprint}
\usepackage[utf8]{inputenc}
\usepackage[full]{textcomp}
\usepackage[osf]{newtxtext}
\usepackage{comment}

\usepackage{amssymb}
\usepackage{mathtools}
\usepackage{hyperref}
\usepackage{breakurl}
\usepackage{mhenvs}
\usepackage{mhequ}
\usepackage{mhsymb}
\usepackage{booktabs}
\usepackage{tikz}
\usepackage{mathrsfs}
\usepackage{longtable}

\usepackage{microtype}
\usepackage{wasysym}
\usepackage{centernot}
\usepackage{enumitem}

\makeatletter
\newcommand{\globalcolor}[1]{%
  \color{#1}\global\let\default@color\current@color
}
\makeatother

\newif\ifdark
\IfFileExists{dark}{\darktrue}{\darkfalse}

\ifdark

\definecolor{darkred}{rgb}{0.9,0.2,0.2}
\definecolor{darkblue}{rgb}{0.7,0.3,1}
\definecolor{darkgreen}{rgb}{0.1,0.9,0.1}
\definecolor{pagebackground}{rgb}{.15,.21,.18}
\definecolor{pageforeground}{rgb}{.84,.84,.85}
\pagecolor{pagebackground}
\AtBeginDocument{\globalcolor{pageforeground}}

\else

\definecolor{darkred}{rgb}{0.7,0.1,0.1}
\definecolor{darkblue}{rgb}{0.4,0.1,0.8}
\definecolor{darkgreen}{rgb}{0.1,0.7,0.1}
\definecolor{pagebackground}{rgb}{1,1,1}
\definecolor{pageforeground}{rgb}{0,0,0}

\fi

\theoremnumbering{Alph}
\renewtheorem{theorem*}{Theorem}

\DeclareMathAlphabet{\mathbbm}{U}{bbm}{m}{n}
\DeclareFontFamily{U}{BOONDOX-calo}{\skewchar\font=45 }
\DeclareFontShape{U}{BOONDOX-calo}{m}{n}{
  <-> s*[1.05] BOONDOX-r-calo}{}
\DeclareFontShape{U}{BOONDOX-calo}{b}{n}{
  <-> s*[1.05] BOONDOX-b-calo}{}
\DeclareMathAlphabet{\mcb}{U}{BOONDOX-calo}{m}{n}
\SetMathAlphabet{\mcb}{bold}{U}{BOONDOX-calo}{b}{n}

\let\epsilon\varepsilon

\def\f{\frac}
\def\1{\mathbf{1}}
\def\E{{\symb E}}

\def\Lip{{\mathrm{Lip}}}
\def\Osc{{\mathrm{Osc}}}
\def\${|\!|\!|}
\def\cst#1{\big(H-{\textstyle {#1\over 2}}\big)}
\def\id{\mathrm{id}}
\def\BC{\mathrm{BC}}
\def\RKHS{\normalfont\textsc{rkhs}}
\def\<{\langle}
\def\>{\rangle}
\setlist{noitemsep,topsep=4pt}
\def\para_#1{/\!\!/_{\!#1}}

\def\restr{\mathord{\upharpoonright}}

\def\slash{\kern0.18em/\penalty\exhyphenpenalty\kern0.18em}
\def\dash{\kern0.18em--\penalty\exhyphenpenalty\kern0.18em}

\makeatletter % Stolen from the internet to make a fat \cdot which isn't as fat as a \bullet
\newcommand*{\fat}{}% Check if undefined
\DeclareRobustCommand*{\fat}{%
\mathbin{\mathpalette\bigcdot@{}}}
\newcommand*{\bigcdot@scalefactor}{.5}
\newcommand*{\bigcdot@widthfactor}{1.15}
\newcommand*{\bigcdot@}[2]{%
  \sbox0{$#1\vcenter{}$}% math axis
  \sbox2{$#1\cdot\m@th$}%
  \hbox to \bigcdot@widthfactor\wd2{%
    \hfil
    \raise\ht0\hbox{%
      \scalebox{\bigcdot@scalefactor}{%
        \lower\ht0\hbox{$#1\bullet\m@th$}%
      }%
    }%
    \hfil
  }%
}
\makeatother

\newtheorem{assumption}[lemma]{Assumption}

\title{Averaging dynamics driven by fractional Brownian motion}
\author{Martin~Hairer and Xue-Mei~Li}

\institute{Imperial College London, UK\\
\email{m.hairer@imperial.ac.uk, xue-mei.li@imperial.ac.uk}}

\begin{document}
\maketitle
\begin{abstract}
We consider slow\slash fast systems where the slow system is driven by fractional Brownian
motion with Hurst parameter $H>{1\over 2}$. We show that unlike in the case $H={1\over 2}$,
convergence to the averaged solution takes place in probability and the limiting process 
solves the `na\"\i vely' averaged equation. Our proof strongly relies on the recently obtained
stochastic sewing lemma. \\[.4em]
\end{abstract}
\setcounter{tocdepth}{2}

\tableofcontents

\section{Introduction}
The purpose of the paper is to study a two-scale stochastic evolution on $\R^d$ with memory
of the type
\begin{equ}
dx_t^\eps = f(x_t^\eps,y_t^\eps)\,dB_t + g(x_t^\eps,y_t^\eps)\,dt\;, \label{e:equationx-1}
\end{equ}
where $B$ is an $m$-dimensional fractional Brownian motion (fBm) of Hurst parameter $H> \f 12$, 
$f: \R^d\times \CY \to L(\R^{m},\R^d)$ and  $g:\R^d\times \CY \to  \R^d$. 
The fast variable $y_t$ is assumed to take values in a state space $\CY$ which is
either an arbitrary Polish space or a compact manifold, depending on the situation.  
We will consider both the case in which the dynamic of $y$ is given, 
independently of that of $x$, and the case in which the current state of $x$ influences the dynamic of $y$.
In the latter case, we will assume that the dynamic of $y$ is Markovian, conditional on $B$.

We recall that one-dimensional fractional Brownian motion is the centred Gaussian process with $B_0 = 0$
and covariance $\E(B_t-B_s)^2=|t-s|^{2H}$.
An $\R^m$-valued fBm $(B_t^1, \dots, B_t^m)$ is obtained by taking 
 i.i.d.\ copies of a one-dimensional  fBm. A fBm
 is not  a semi-martingale and does not have independent increments.
It does however have a version such that almost all of its sample paths $t\mapsto B_t(\omega)$
 are H\"older continuous of order $\alpha$ for any $\alpha<H$.

Let us first consider the simple case in which the fast process $y$ has no feedback from $x$ and 
is of the form $y^\eps_t = Y_{t/\eps}$ for some process $Y$  
which is almost surely H\"older continuous of order $\alpha$ with $\alpha + H > 1$. 
The integral  appearing in \eqref{e:equationx-1} can then be interpreted as a Young integral. 
For the processes $\{ x^\epsilon_{\fat}, \epsilon>0\} $ to have a limit, 
we would at the very least need uniform bounds. The usual Young bound however 
only gives an estimate of the form
$$\Big|\int_0^t f(x_s^\epsilon, y_s^\epsilon) db_s\Big|\lesssim |f(x_{\fat}^\epsilon, y_{\fat}^\epsilon)|_\alpha |b|_\beta\;,$$
which is not very helpful since the process $y^\epsilon$ is in general 
expected to have a H\"older norm of order $\f 1 {\epsilon^\alpha}$. 
Proving these bounds present unexpected difficulties. In the case where $B$ is a Brownian motion,
the desired estimates follow quite easily from It\^o's isometry and / or the 
Burkholder--Davis--Gundy inequality. They are of course not available in our setting,
but we would nevertheless like to exploit the stochastic nature of the fractional Brownian motion. 
We resolve this problem by using a carefully chosen approximation to the Young integral and using a recently discovered stochastic sewing lemma by Lê \cite{Khoa}. Our main result in this setting
is given by Theorem~\ref{weak-limits} below. When combining it with Lemma~\ref{le:crucial},
this can be formulated as follows.

\begin{theorem*}\label{thm:indep}
Let $\CY$ be a Polish space and let $f,g$  be bounded measurable and of class $\BC^2$ in their first argument. Let $y^\eps_t = Y_{t/\eps}$ for a $\CY$-valued stationary stochastic
process $Y$ that is independent of $B$ and is strongly mixing with rate $t^{-\delta}$ for
some $\delta > 0$ in the sense that
\begin{equ}
\sup\big\{\P (A \cap \bar A) -  \P(A)\, \P(\bar A)\,:\, A\in \sigma(Y_0), \bar A\in \sigma(Y_t)\}
\lesssim t^{-\delta}\;.
\end{equ} 
Let $\bar f(x) = \int f(x,y)\,\mu(dy)$, where $\mu$ is the law of $Y_0$, and
similarly for $g$.
Then, any solutions to \eqref{e:equationx-1} converge in probability to the solution to
\begin{equ}[e:averagedSDE]
d\bar x_t = \bar f(\bar x_t)\,dB_t + \bar g(\bar x_t)\,dt\;,
\end{equ}
with the same initial value.
\end{theorem*}

\begin{remark}
This is very different from the case where $B$ is a Wiener process. In that case, one cannot expect 
convergence in probability and the weak limit solves a diffusion with averaged generator, see e.g.
\cite{Hasminski, KLO,FW,Kifer,Skorohod-Hoppensteadtt-Salehi,Pavliotis-Stuart, averaging, Li}, which
is different in general from the diffusion with averaged diffusion coefficients appearing here. 
In this sense, equations driven by fBm with $H > \f12$ behave more like ODEs rather than SDEs.

Note however that our convergence in probability refers to convergence in probability in the 
full `product' probability space on which \textit{both} $B$ and $Y$ live. In particular,
we do not know whether the convergence to $\bar x$ holds with $B$ replaced by any given 
$b \in \CC^\beta$ with $\beta < H$. In the lingo of diffusions in random environment,
our convergence result is `annealed' rather than `quenched'.
\end{remark}

\begin{remark}
It would be natural to take for $y$ the solution to an SDE driven by a fractional Brownian motion
independent of $B$. Unfortunately, it is not clear whether the results of 
\cite{Additive,Hormander,Rough} concerning the ergodicity of such processes can be strengthened in
order to satisfy the assumptions of Theorem~\ref{thm:indep}.
\end{remark}

The other case we consider is when the state of the slow variable $x$ feeds back into
the dynamic of the fast variable $y$. 
In this case, we restrict ourselves to the case when $\CY$ is a compact Riemannian manifold and $y$
is given by the solution to 
\begin{equ}
dy_t^\eps = {1\over \eps} V_0(x_t^\eps,y_t^\eps)\,dt + {1\over \sqrt \eps}V(x_t^\eps,y_t^\eps)\circ d\hat W_t\;, 
\label{e:equationy-1}
\end{equ}
where, for any fixed value $x\in \R^d$,  $V_i(x,\fat)$ are vector fields on a state space $\CY$, and where $\hat W$ is an $\hat m$-dimensional
standard Wiener process that is independent  of $B$.

Since solutions to this equation are expected to be H\"older continuous of any order $\alpha<\f 12$, 
the integral with respect to $B$ appearing in \eqref{e:equationx-1} can still be interpreted as a Young 
integral for any fixed $\eps > 0$. Since the slow and fast variables interact with each other however, 
a solution theory with mixed Young and It\^o integrals must be used.  
Such a theory is available in the literature, see for example the work by 
Guerra--Nualart \cite{Guerra-Nualart} extending 
Kubilius \cite{Kubilius}, as well as \cite{dasilva-Erraoui-ElHassan}.  
Our main theorem is the following result, which is a slight reformulation of Theorem~\ref{thm:main} below.

\begin{theorem*} \label{thm:feedback}
Let $f,g$ and the $V_i$ satisfy Assumption~\ref{ass:main} below, 
let $B_t$ be a fBm of Hurst parameter $H>\f 12$,
and let $\hat W_t$ be an independent Brownian motion.
For every $x \in \R^d$, let $\mu^x$ denote the (unique) invariant measure 
for \eqref{e:equationy-1} with $x^\eps_t$ replaced by $x$.
As before, let $\bar f(x) = \int f(x,y)\,\mu^x(dy)$, and similarly for $\bar g$.

Then, as $\epsilon\to 0$,
the process $x_t^\eps$ converges in probability in $\CC^\alpha$ (for any $\alpha<H$)  
to the unique limit $\bar x_t$ solving \eqref{e:averagedSDE} with the same initial value.
\end{theorem*}

\subsection{Outline of the article}

In both cases, the proof of convergence of the slow variable is based on a deterministic residue bound, 
Lemma~\ref{general-convergence}. This is a quite general statement about differential equations with 
Young integration,  not involving any stochastic element nor needing any preparation,  
and  is therefore given at the very beginning of the article, 
even though it becomes relevant only in the later stages.

Section~\ref{section-no-feedback} is devoted to the proof of Theorem~\ref{thm:indep}. In order
to prepare for this, we use  the Mandelbrot--Van Ness representation of $B_t$ by a Wiener process $W_t$. 
We then make use of the  observation,  \cite{Additive}, that  the filtration $\CG_t$ generated by the increments of $B$ up to time $t$ is the same as that generated by $W$, 
and that $B$ can be decomposed for $t>0$ as $B_t=\bar B_t+\tilde B_t$, where $\bar B_t$ is 
smooth in $t$ and $\tilde B_t$ is 
independent of $\CG_0$. Such a split can be made with reference to any $\CG_u$ for any time $u$,
which allows us to define integrals of the type
\begin{equ}
\int_u^v F(s)\,dB(s)\;,
\end{equ}
for $\CG_u$-measurable processes $F$, as the sum of a Wiener integral against $\tilde B_t$
and a Riemann--Stieltjes integral against the smooth function $\bar B_t$.
The stochastic sewing lemma then allows us to extend this integration to a class of adapted integrands
which are allowed to be quite singular (much more than what Young integration would allow), but such 
that the singular part of their behaviour is independent of $B$ in a suitable sense.
This is the content of Lemma~\ref{lem:mainTight}, which is the main ingredient
of the proof of Theorem~\ref{thm:indep} given in Section~\ref{sec:semidet}.

%The representations allow us to obtain the basic estimate Lemma~\ref{lemma-a},
%with which and the sewing lemma \cite{Max}, we observe that  the composition map 
% from $\CC^{-\kappa, \gamma}\times \CC^{\alpha}\to \CC^{-\kappa}$ is continuous and 
%finally built the estimates (Lemma \ref{lem:mainTight})  for the convergence statement in Theorem~\ref{thm:indep}.

Section~\ref{sec:averaging}
is devoted to the proof of Theorem~\ref{thm:feedback}. The main ingredient of the proof is 
given by Theorem \ref{fixed-point-map} where we show that one has a bound
 $$
\left\|\int_s^t(  h(x_r, y_r^\eps)- \bar h(x_r)) \, dB_r \right\|_{L^p} \leq
 C\,  \eps^\kappa   \left( \|x\|_{\alpha, p} |t-s|^{\bar \eta} 
+ |t-s|^{\eta}\right)\;,
$$
for some $\eta > {1\over 2}$ and $\bar \eta > 1$,
where $\bar h$ is the average of $h$. 
Compared to the results in Section~\ref{section-no-feedback}, the difficulty here is that the process $y^\eps$
does depend on $x$ (and therefore also on $B$) via \eqref{e:equationy-1}.
The main idea is to interpret the integral appearing in this expression as the output of the stochastic
sewing lemma applied to
\begin{equ}
A_{u,v} = \int_u^v(  h(x_u, Y_r^{x_u,\eps})- \bar h(x_u)) \, dB_r\;,
\end{equ}
where $Y_r^{\bar x,\eps}$ denotes the solution to \eqref{e:equationy-1}, but 
with the process $x$ replaced by the fixed value $\bar x$. In this way, the integrand is
$\CF_u$-measurable (for $\CF$ the filtration generated by $B$ and $\hat W$) and the integral
can be interpreted as a mixed Wiener / Young integral as before.

The hard part is to show that $\delta A$  satisfies the assumptions of the stochastic sewing lemma.
For this, we use the fact that we only need bounds on  $\E \bigl(\delta A_{s,u,t}\,|\,\CF_s\bigr)$
and that this quantity is much better behaved than $\delta A$ itself. 
Section~\ref{sec:semigroup} contains preliminary estimates on the Markov semigroup
generated by $Y_r^{\bar x,\eps}$ as well as some form of `non-autonomous Markov semigroup'
generated by \eqref{e:equationy-1}, while Section~\ref{sec:uniformBounds} then contains the uniform
bounds on the conditional expectation of $\delta A$.

%  \begin{equ}
%|\CP_{u-s}^{x_s}F - \CQ_{s,u}^x F|_\infty 
%\lesssim   |F|_\Lip   |u-s|^\alpha\,\sqrt{\E (|x|_\alpha^2 \,|\,\CF_s \vee \CG_u)}\;,
%\end{equ}
%stochastic sewing lemma and we  Lemma~\ref{lemma-a} 
%   Lemma~\ref{lemma-semi-group-0}, Lemma~\ref{lem:boundNegHolderGen},
% and Lemma \ref{lemma:QP}. The last three Lemmas are of the nature of ergodic theorems
% and approximations of the fast variables.

% \begin{lemma}\label{lem:boundNegHolderGen}
%Fix $p\ge 2$, $x \in \R^d$ and $y \in \CY$. Let $F \colon \CY \to \R$ be bounded measurable and write
%$\bar F (x)= \int F(y)\,\mu^x(dy)$. 
%Then, 
%\begin{equ}
%\Big\|\int_0^t \big(F(\bar \Phi_{0,t}^{x}(y)) - \bar F(x)\big)\,dr\Big\|_p
%\lesssim \eps^{1 \over p}\, t^{1-{1\over p}}|F|_\Osc\;,
%\end{equ}
%holds uniformly over $x \in \R^d$, $y \in \CY$, and $t \ge 0$.
%\end{lemma}

\subsection*{Notation}
We gather here the most common notations. 
\begin{itemize}
\item $(\Omega, \CF,  \P)$ is a  probability space and $\|\fat\|_{p}$ denotes the norm in $L^p(\Omega)$. 
\item For $s \le t$ and $x_t$ a one-parameter process with values in $\R^d$, we set $\delta x_{s,t} \eqdef x_t - x_s$.
We also set $\|x\|_{\alpha,p} = \sup_{s,t} |t-s| ^{-\alpha}\|\delta x_{s,t}\|_{p}$.
\item For $s<u<t$ and $A$ a two-parameter stochastic process, we set 
\begin{equ}
\delta A_{sut} \eqdef A_{s,t} - A_{s,u} - A_{u,t}\;.
\end{equ}
We also set 
\begin{equ}
\|A\|_{\alpha,p} \eqdef \sup_{s < t} \f{ \| A_{s,t} \|_{p} }  {|t-s|^{\alpha}}\;,\qquad 
\$A\$_{\alpha,p} \eqdef \sup_{s < u < t} \f{ \|\E(\delta A_{sut}\,|\,\CF_s) \|_{p} }  {|t-s|^{\alpha}}\;.
\end{equ}
\item $H_\eta^p=\{ A_{s,t}\in L_p(\Omega, \CF_t,\P): \|A\|_{\eta,p}<\infty\}$.
\item $\CB_{\alpha, p}=\{ x_t \in\CF_t: \delta x_{s,t}\in H^p_\alpha\}$, $\CB_{\alpha, p} \subset L_p(\Omega,\CC^\gamma)$ (up to modification) for  $\gamma<\alpha -\f 1p$.
\item $\bar H_\eta^p=\{ A_{s,t}:  \$A\$_{\eta,p}<\infty\}$.

\item $W_t$ and $\hat W_t$ are two independent two-sided Wiener processes of dimension $m$ and $\hat m$ respectively. 
\item $\CG_t$ and $\hat \CG_t$ are the filtrations generated by the independent Wiener processes $W$ and $\hat W$ respectively and 
${\CF_t=\CG_t\vee \hat \CG_t}$.
\item $B_t$ (also denoted  by  $B_t^H$)  is a fractional Brownian motion of Hurst parameter $H$, which is related to $W_t$ via the Mandelbrot-Van Ness representation.
\item For $u<t$, $\bar B^u_{t} \eqdef \E \bigl(B_t-B_u\,|\, \CG_u\bigr)$ and $\tilde B^u_{t}\eqdef B_t-B_u-\bar B^u_{t} $. \\
Also $\bar B_t\eqdef \bar B_t^0$, $\tilde B_t=\tilde B_t^0$.
\item $f \lesssim  g$ means that $f\le Cg$ for a universal constant $C$.
\item $\CC_K^\infty$ denotes the space of smooth functions with compact support.
\item $\BC^k$  is the space of bounded $\CC^k$ functions with bounded derivatives of all orders up to $k$.
\item For $\alpha \in (0,1)$, $|x|_\alpha =\sup_{s\neq t} \f{| x_t-x_s|}{|t-s|^\alpha}$ is the homogeneous H\"older semi-norm. 
\item  $|\fat|_\infty $ and $|\fat|_\Lip$ denote the supremum norm and minimal Lipschitz constant respectively. 
\item $|f|_\Osc=\sup f-\inf f$.
\item For $h \in\CC(\R,\R^d)$, $\kappa\in (0,1)$,
$|h|_{-\kappa} = \sup_{s,t \le T} |t-s|^{\kappa-1} \bigl|\int_s^t h(r)\,dr\bigr|\;$.
\item For $f: \R\times \R^d \to \R$, $|f|_{-\kappa, \gamma}$ is the smallest possible choice of constant
$K$ with the property
\begin{equs}
\sup_x |f(\fat,x)|_{-\kappa} \le K\;,\quad 
\sup_{x\neq y} {|f(\fat,x) - f(\fat,y)|_{-\kappa} \over |x-y|^{\gamma}} \le K\;. 
\end{equs}
\item We write $\CB(\CY)$ for the Borel $\sigma$-algebra of a topological space $\CY$. 
For $s<u$, we set $$\CU_s^u=\{ F: \Omega\times \CY \to \R : \hbox { bounded $(\CF_s\vee \CG_u)\otimes \CB(\CY)$ measurable} \}.$$

\end{itemize}

\subsection*{Acknowledgements}
{\small
We would like to thank Massimiliano Gubinelli for an inspiring conversation.
MH gratefully acknowledges financial support from the Leverhulme trust via a Leadership Award
and the ERC via the consolidator grant 615897:CRITICAL.
}

\section{A deterministic residual bound}
\label{sec:residue}

We first state a bound on the difference between two solutions to a differential equation
driven by a Hölder continuous signal, given a bound on the corresponding residual. 
In the following, the reader may think of $b_t $ as a realisation of
the fractional Brownian motion $B_t$ or a realisation of $(B_t,t) \in \R^{m+1}$, but our statement
is purely deterministic.

A basic tool is the following estimate, the proof of which is elementary and follows for example
easily from \cite[Eq.~2.8]{Composition}.
\begin{lemma}\label{lem:Lip}
Assume that $F: \R^d \to \R$ has two bounded derivatives and let $\alpha \in (0,1)$.
Then, the composition operator  $x \mapsto F(x) = (t\mapsto F(x_t))$ satisfies the bound
\begin{equ}
|F(x)-F(y)|_\alpha \lesssim |F'|_\infty |x-y|_\alpha +  |F''|_\infty |x-y|_\infty (|x|_\alpha + |y|_\alpha)\;.
\end{equ}
\end{lemma}

The announced bound goes as follows.

\begin{lemma}\label{general-convergence}
Let $F \in \BC^2$, let $b\in \CC^\beta$ for some $\beta > {1\over 2}$ and
let $Z, \bar Z \in \CC^\alpha$ for some $\alpha \in (0,\beta]$ such that $\alpha + \beta > 1$. Let $z, \bar z$ be the solutions to
\begin{equ}
z_t = Z_t + \int_0^t F(z_s)\,db_s \;,\qquad 
\bar z_t = \bar Z_t + \int_0^t F(\bar z_s)\,db_s \;.
\end{equ}
Then, there exists a constant $C$ depending only on $F$ such that,
on the time interval $[0,1]$, one has the bound
\begin{equ}
|z-\bar z|_{\alpha} \le C\exp\left(C | b|_\beta^{1/\beta} + C |Z|_\alpha^{1/\alpha} + C |\bar Z|_\alpha^{1/\alpha}\right) |Z-\bar Z|_{\alpha}\;.
\end{equ}
\end{lemma}

\begin{proof}
Since, by Lemma~\ref{lem:Lip}, we have the bound
\begin{equ}[e:boundDiffF]
|F(z)-F(\bar z)|_\alpha \lesssim  |z-\bar z|_\alpha |F'|_\infty +  |F''|_\infty |z-\bar z|_\infty (|z|_\alpha + |\bar z|_\alpha)\;,
\end{equ}
we conclude that on $[0,T]$ with $T \le 1$ one has
\begin{equs}
|z-\bar z|_\infty &\lesssim \big(T^{\beta}L |z-\bar z|_\infty + |F'|_\infty T^{\alpha + \beta} |z-\bar z|_\alpha \big) |b|_\beta
+ |Z-\bar Z|_\infty\;,\\
|z-\bar z|_\alpha &\lesssim  \big(T^{\beta-\alpha }L |z-\bar z|_\infty +  |F'|_\infty T^{\beta} |z-\bar z|_\alpha  \big) |b|_\beta
+ |Z-\bar Z|_\alpha\;,
\end{equs}
where $L=  |F|_\Lip+ T^\alpha |F''|_\infty (|z|_\alpha + |\bar z|_\alpha)$.
The two inequalities are proved similarly, we demonstrate with the second one. By \eqref{e:Young} in \S~\ref{section-no-feedback},
we obtain on $[0,T]$ the bound
\begin{equs}
|z-\bar z|_\alpha &\leq
   \Big| \int_0^{\fat} (F(z_s)-F(\bar z_s)) \,db_s\Big|_\alpha+| Z-\bar Z|_\alpha\\
 &  \lesssim |F(z)-F(\bar z)|_\alpha T^{\beta} |b|_\beta +|F(z)-F(\bar z)|_\infty T^{\beta-\alpha} |b|_\beta+ | Z-\bar Z|_\alpha\;,
\end{equs}
and the requested bound then follows from \eqref{e:boundDiffF}.
In a similar way, using the fact that $|F(z)|_\alpha \le |F'|_\infty |z|_\alpha$, we obtain the a priori bound
\begin{equ}
|z|_\alpha \lesssim |Z|_\alpha + T^{\beta-\alpha} |b|_\beta + T^\beta |b|_\beta |z|_\alpha\;,
\end{equ}
and similarly for $\bar z$. Provided that we choose $T$ in such a way that
\begin{equ}[e:boundT]
T^\beta  |b|_\beta \le c\;,\qquad T^\alpha |Z|_\alpha \le 1\;,\qquad T^\alpha |\bar Z|_\alpha \le 1\;,
\end{equ}
for some sufficiently small constant $c$ that only depends on $F$, we thus obtain the 
bound $|z|_\alpha \lesssim  |Z|_\alpha + T^{\beta-\alpha} |b|_\beta$, and similarly for $|\bar z|_\alpha$. 
In particular, this
shows that for $T$ as in \eqref{e:boundT} one has $L \lesssim 1$.
  
This then suggests the introduction of the norm
\begin{equ}
|z|_{\alpha,T} = |z|_\infty + T^\alpha |z|_\alpha\;, 
\end{equ}
with suprema taken over $[0,T]$, for which we obtain the bound
\begin{equ}
|z-\bar z|_{\alpha,T} \lesssim  T^{\beta} |z-\bar z|_{\alpha,T} |b|_\beta
+ |Z-\bar Z|_{\alpha,T}\;,
\end{equ}
thus yielding
\begin{equ}
|z-\bar z|_{\alpha,T} \le 2|Z-\bar Z|_{\alpha,T}\;,
\end{equ}
on $[0, T]$ where $T$ is as in \eqref{e:boundT}. Iterating this bound, we conclude that on any sub-interval
$[s,s+T]$ of $[0,1]$ one has a bound of the type
\begin{equ}
|(z-\bar z) \restr [s,s+T]|_{\alpha,T} \le 2\exp(C (1+s/T)) |Z-\bar Z|_{\alpha,T}\;,
\end{equ}
whence we conclude that on $[0,1]$,  for a possibly larger constant $C$, one has 
\begin{equ}
|z-\bar z|_{\alpha} \lesssim \exp(C(1+T^{-1})) |Z-\bar Z|_{\alpha}\;.
\end{equ}
Since \eqref{e:boundT} allows us to choose $T$ such that $1/T \lesssim |b|_\beta^{1/\beta} + |Z|_\alpha^{1/\alpha} + |\bar Z|_\alpha^{1/\alpha}$,
the claim follows.
\end{proof}

\section{Averaging without feedback}
\label{section-no-feedback}

In this section we provide an interpretation of the integral against fractional Brownian motion
that is more stable than the Young integral in situations in which the integrand exhibits
fast oscillations. The idea is to exploit the adaptedness of the integrand in a way that allows
us to apply the stochastic version of the sewing lemma \cite{Max} recently obtained in  \cite{Khoa}.

To take one step back, we recall that integration of a deterministic function with respect to a 
fractional Brownian motion (fBm) of Hurst parameter $H>\f 12$ is
called a Wiener integral (with respect to Gaussian processes): the integrands are smooth stochastic processes 
completed with the norm given by the inner product 
\begin{equ}
\< \phi, \psi \> =\E \Big(\int_\R \phi_s dB_s \int_\R \psi_s dB_s\Big)\;.
\end{equ}
(Limits of smooth functions 
with respect to this norm can be Schwartz distributions.)
When the integrand is sufficiently smooth, this is just the Young integral.

Let $B_t$ be an $m$-dimensional fractional Brownian motion with Hurst parameter $H\in (\f 12, 1)$, 
it has an integral representation with respect to a two sided standard Wiener process $W_t$,
which was introduced by Mandelbrot and Van Ness \cite{FBM}. We consider $H$ as being fixed throughout this article 
and therefore omit the superscript.
For $r > u$,  write the increment of fractional Brownian motion as a sum of two processes:
\begin{equs}
B_r-B_u
&= \int_{-\infty}^u \Big( (r-v)^{H-{1\over2}} - (u-v)^{H-{1\over 2}} \Big)\,dW_{v}
+ \int_u^r (r-v)^{H-{1\over2}}\,dW_{v} \\
&\eqdef \bar B^u_{r} + \tilde B^u_{r}\;.\label{e:decompB}
\end{equs}
%where $W$ is a standard Wiener process.
Writing $\CG_t$ for the filtration generated by the increments of $W$,
$\bar B^u_{t}$ is $\CG_u$-measurable and smooth in $t$ on $(u, \infty)$, while 
$\tilde B^u_{t}$ is independent of $\CG_u$. For the special case
$u=0$, we simply write $ \bar B_{t}= \bar B^0_{t}$,  $\tilde B_t=\tilde B^0_{t}$.  
Recall also that the filtration  $\CG_t$ coincides with that  generated  by the increments 
of $B$.

% There are  several approaches to stochastic calculus with respect to Brownian motions, which we do not need, see for example
%\cite{ Duncan-Hu-Pasik-Duncan}. 
%\cite{Alos-Mazet-Nualart}??\cite{Decreusefond-Ustunel}?

\subsection{Mixed Riemann and Wiener integrals}
\label{sec:basicEstimate}

If  $f \colon \R \times \R^d \to L(\R^m, \R^d)$  is a measurable function and  $x_t$ a $\CG_t$-adapted stochastic process,
our first task is to define $\int_0^t f(r, x_r) \, dB_r$ as the limit of `Riemann sums' of the type
$\sum_i \int_{s_i}^{s_{i+1}} f(r, x_{s_i})\,dB_r$, provided that $f$ and $x$ satisfy suitable assumptions. 
Prior to justifying  its convergence we
explain  how each individual integration  in the sum is defined.  For any $s<t$, set
\begin{equ}
A_{s,t} \eqdef  \int_s^t f(r, x_s)\,dB_r \eqdef  \int_s^t f(r,x_s)\,d\bar B^s_r+\int_s^t f(r,x_s)\, d \tilde B^s_{r}\;.
\end{equ}
 The first integral will be considered as a Riemann--Stieltjes integral which will exploit the fact that 
 $\bar B^s$ is a smooth function with a well-behaved singularity at time $s$.
The second term will be interpreted as a Wiener integral with respect to the Gaussian process $\tilde B^s$,
which we can do since $x_s$ is $\CG_s$-measurable and therefore independent of it. 
Since $r\mapsto \bar B_r^u$ is smooth for $r > u$ and its derivative has an integrable singularity at $r \sim u$,
the Riemann integral $\int_u^t f(r,x_u)\,d\bar B_r^u$  can be defined in a pathwise sense as soon as $f$ is
continuous in both of its arguments.
 If $x_{\fat}$ has continuous sample paths, then the same is true for the Wiener integral since 
 the map $F \mapsto \int_u^t F_r\,d\tilde B^u_r$,
viewed as a linear map from $\CC^\infty$ into $L^2(\Omega)$, can be extended to all $F \in \CC^0$
(and actually even to $F \in \CC^{-\kappa}$ for $\kappa$ small enough, see Lemmas~\ref{lemma-covariance} and~\ref{lemma-kernel} below).
Think now of $u$ as being fixed and consider an arbitrary stochastic process $F$ on $[u,t]$,
but we think of the case $F_r = f(r,x_u)$.

\begin{remark}
If $F$ is either deterministic and H\"older continuous of order $\alpha$ or $F\in \CB_{\alpha, p}$
where $p>2$ and $\alpha+H>1$, then the mixed integral coincides with the Young integral. 
The first follows from the deterministic 
sewing lemma and that $\int_u^t (F_r-F_u)\,  d\tilde B_r \lesssim |t-u|^{\alpha+H}$. The second follows from 
the stochastic sewing lemma, alternatively this is a special case of  Lemma~\ref{lem:sewing=Young} below.
\end{remark}

In situations where $f$ and $x$ are sufficiently regular so that 
the usual Riemann sums converge, we will see in Lemmas~\ref{lem:sewing=Young} and~\ref{prop:equalA} 
that the notion of integration used
here coincides with the classical Young integral. The advantage of this set-up however
is that we can exploit the stochastic cancellations of the Wiener integral through the
use of the stochastic sewing lemma, which allows us to substantially expand the class of admissible
integrands and is fundamental for extracting uniform estimates for SDEs with random inputs.

We begin with building up estimates for the mixed stochastic integral explained earlier.
Let  $R$ denote the covariance function of $\tilde B$.  We work componentwise, so that instead 
of complicating our notation
with i.i.d.\ copies of the one dimensional fBm's, we may  assume that 
$\tilde B$ is one dimensional in the formulation below. 
It follows from the scaling properties of $\tilde B$
that 
\begin{equs}
R(r,s) &=\E \tilde B_r \tilde  B_s= (r\wedge s)^{2H} \hat R \Big({|r-s| \over r\wedge s}\Big)\;,\\
\hat R(t) &=\E \tilde B_1 \tilde B_{1+t} = \int_0^1 (1-s)^{H-{1\over 2}}(1+t-s)^{H-{1\over 2}}\,ds\;, \label{e:defFR}
\end{equs}
so that their distributional derivatives $\d^2_{r,s} R(r,s)\eqdef \f{\partial^2} {\partial r \partial s}  R(r,s)$ 
satisfy
 \begin{equs}
\d^2_{r,s} R(r,s) &= (r\wedge s)^{2H-2} G\Big({|r-s| \over r\wedge s}\Big)\;, \label{e:relRG}\\
G(t) &= (2H-1) \hat R'(t) - (t+1) \hat R''(t)\;.
\end{equs}

\noindent{\bf Convention.}
We now fix a filtration $\CF_s$ with $\CG_s \subset \CF_s$ and such that, for every~$s$,
$\tilde B^s$ is independent of $\CF_s$. The example to have in mind which will be relevant in Section~\ref{sec:averaging} 
is to take $\CF_s = \CG_s \vee \hat \CG_s$,
where $\hat \CG$ is the filtration generated by the increments of a Wiener process independent of
$B$.

Recall that Wiener integrals are centred Gaussian processes. In our case, $F_s$ is random but  with $F_s\in \CF_u$ for any $s\in [u,t]$,
so that $\int_u^t  F_s d\tilde B_s^u$ 
is a centred Gaussian process, conditional on $\CF_u$.
\begin{lemma}\label{lemma-covariance}
$G(t) \approx t^{2H-2}$ for $t \ll 1$ and
$G(t) \approx t^{H-{3\over 2}}$ for $t \gg 1$.
In particular, $\d^2_{r,s} R(r,s)$ is integrable over any bounded region and there exists  $c_1 \in \R$ s.t.
\begin{equ}
 \int_0^t \int_0^t |\d^2_{r,s} R(r,s)|\,dr\,ds \le c_1 t^{2H}\;,
\end{equ}
for every  $t\in \R_+$.
%The notation $\lesssim $ means that the left hand side is less or equal to a multiple of the right hand side.
For some fixed $u \ge 0$, let $F_s$ be pathwise smooth in $s$ and suppose  $F_s$ is $\CF_u$-measurable for any $s\in [u,t]$, 
then the following It\^o isometry holds, 
\begin{equ}[e:RKHS]
 \E \left( \Big(\int_u^t F_s\,d\tilde B_s^u\Big)^2 \, \Big| \, \CF_u\right)(\omega)
= \int_u^t \int_u^t \d^2_{r,s} R(r,s)
F_r(\omega)F_s(\omega) \,dr \,ds\,.
\end{equ}
\end{lemma}

\begin{proof}
The bound on $\d^2_{r,s} R(r,s)$ follows at once from the representation given in 
Lemma~\ref{lem:R} below, while the fact that \eqref{e:RKHS} holds in the distributional sense is
classical, see for example \cite{WienerInt}. The bounds on $R$ given in Lemma~\ref{lem:R} 
guarantee that the distributional derivative of $R$ coincides with its weak derivative and
is an integrable function, so that \eqref{e:RKHS} also holds with the right hand side interpreted
as a Lebesgue integral.
\end{proof}

As usual \cite{WienerInt,Bogachev}, the left hand side of \eqref{e:RKHS} can be defined in
such a way that this isometry 
extends to all $g$ taking values in the completion of the 
space of smooth functions under the norm given by the right hand side of \eqref{e:RKHS}. 
This in particular contains all $g \in L^2$, as can be shown similarly to \cite{Nualart}. 
It turns out however that the space of admissible integrands for this Wiener integral contains
not only functions, but also distributions of order $-\kappa$ provided that $\kappa < H-{1\over 2}$.
More precisely, we have the following result, a proof of which is 
postponed to the appendix.

\begin{lemma}\label{lemma-kernel}
Let $h$ be a continuous function. Then, for all $T>0$ and $\kappa \in [0,H-{1\over 2})$,
one has the bound
\begin{equ}[e:RKHScond]
|h|_{\RKHS}^2 \eqdef \Bigl| \int_0^T \int_0^T \d^2_{r,s} R(r,s)\, h(r)h(s)\,dr\,ds\Bigr| \lesssim T^{2H-2\kappa } |h|_{-\kappa}^2\;,
\end{equ}
where the negative H\"older norm $ |h|_{-\kappa}$ on $[0,T]$ is given by
\begin{equ}
|h|_{-\kappa} = \sup_{0\le s,t \le T} |t-s|^{\kappa-1} \Bigl|\int_s^t h(r)\,dr\Bigr|\;.
\end{equ}
\end{lemma}

Since $\bar B_t$ is smooth in $t$, integrals with respect to it extend to rougher integrands, 
as we will show now.   Below we provide a bound for integration with respect to the full fBm.
%which is used for Lemmas~\ref{lem:boundAst} and~\ref{lem:sewing=Young}.
\begin{lemma}
\label{lemma-a}
Let $B$ be a fBm with $H > {1\over 2}$ and fix 
$0 \le \kappa < H-{1\over 2}$. Let $s\ge 0$ be fixed. 
Let $r \mapsto F_r$ be smooth, with each  $F_r$ for $s \le r$ measurable with respect to $\CF_s$.
Then, for $t \ge s$ with $|t-s| \le 1$ and $2 \le p < q$ one has the bound
\begin{equ}
\Big\|\int_s^t F_r\,dB_r\Big\|_p \lesssim \| | F|_{-\kappa}\|_q \,|t-s|^{H-\kappa}\;,
\end{equ}
where  $|F|_{-\kappa}$ denotes its negative Hölder norm on $[s,t]$.

By linearity and density, this immediately allows us to extend the notion of integral against $B$ to 
any integrand in $L^q( (\Omega,\CF_r), \CC^{-\kappa})$ for any $0\le \kappa<H-\f 12$ (which may no longer agree with the Young integral).
\end{lemma}

\begin{proof}
 Since our set-up is translation invariant, we restrict 
ourselves to the case $s=0$ without loss of generality and we write
\begin{equ}
\int_0^t F_r\,dB_r = \int_0^t F_r\,d\bar B_r + \int_0^t F_r\,d\tilde B_r
\eqdef I_1 + I_2\;.
\end{equ}
To bound $I_1$, we note that  $r\mapsto {\bar B_r}(\omega)$ is a smooth function on $(0, \infty)$  
satisfying the bounds, for any $p\ge 1$,
\begin{equ}[e:boundBbar]
\|\dot {\bar B}_r\|_p \lesssim r^{H-1}\;,\qquad
\|\ddot {\bar B}_r\|_p \lesssim r^{H-2}\;.
\end{equ}
We then integrate $I_1$ by parts, so that 
\begin{equ}
I_1 = \dot {\bar B}_t \int_0^t F(r)\,dr -\int_0^t \int_0^r F_u\,du\,\ddot{\bar B}_r\,dr \;,
\end{equ}
and the required bound follows from Hölder's inequality.

Concerning $I_2$, since the integrand is $\CF_0$-measurable and $\tilde B$ is 
independent of $\CF_0$, the Wiener integral $I_2$ is Gaussian and its $L_p$ norm is bounded by its $L_2$ norm.
 We can proceed as if the integrand were deterministic and use Lemma~\ref{lemma-covariance}, so that 
 one has the bound
\begin{equ}[e:boundA2]
\E |I_2|^p   \lesssim \E \Big|\int_0^t\int_0^t F(r)\,F(s)\, \d^2_{r,s} R(r,s)\,dr\,ds\Big|^{p/2}\;.
\end{equ}
Inserting the bound from Lemma~\ref{lemma-kernel}  into \eqref{e:boundA2}, 
we obtain the bound $\|I_2\|_p \lesssim \||F|_{-\kappa}\|_p \;t^{H-\kappa}$ as required.
\end{proof}

\subsection{Stochastic sewing lemma}

Let $A_{s,t}$ denote a two parameter stochastic process with values in $\R^n$,  where $s\le t$. 
Both $s $ and $t$  take values in a fixed finite interval $[a,b]$.  
We are interested in situations where $A$ is close to being an increment. To quantify this,
for any $s< u< t$, set
\begin{equs}
\delta A_{sut} &\eqdef A_{s,t} - A_{s,u} - A_{u,t}\;,
\end{equs} 
which vanishes if and only if $A_{s,t}$ is the increment of a one-parameter function.
In the cases of interest to us, the family of `defects'  $\delta A_{sut}$ is typically 
much smaller than  $A_{s,t}$ itself for $|t-s|$ small.

Let us now quantify this more precisely.
Given $p\ge 2$ and an exponent $\eta>0$, we define the space $H_\eta^p$ 
of continuous functions $(s,t) \mapsto A_{s,t} \in L^p(\Omega,\CF_t)$ such that
\begin{equ}[e:supbound]
\|A\|_{\eta,p} \eqdef \sup_{s < t} \f{ \| A_{s,t} \|_{p} }  {|t-s|^{\eta}} < \infty\;,
\end{equ}
where $\|\fat\|_p$ denotes the norm in $L^p(\Omega)$.
We also define the space $\bar H_\eta^p$ 
of maps $A_{s,t}$ as above such that 
\begin{equ}[e:deltabound]
\$A\$_{\eta,p} \eqdef \sup_{s < u < t} \f{ \|\E(\delta A_{sut}\,|\,\CF_s) \|_{p} }  {|t-s|^{\eta}} < \infty\;.
\end{equ}
Then $H_\eta^p$ is a Banach space with norm $ \|\fat\|_{\eta,p}$, while $\$\fat \$_{\eta,p}$ is only a semi-norm.

We will view a partition of the interval $[a,b]$ as a collection $\CP$ of non-empty 
closed intervals that cover $[a,b]$ and
overlap pairwise in at most one point, so we can use the notation
$\sum_{[u,v] \in \CP} A_{u,v}$ for the `Riemann sum' associated with $A$ on the partition $\CP$.
Given such a partition, we write $|\CP|$ for the length of the largest interval contained in $\CP$.
The following result was proved in  \cite[Thm~2.1]{Khoa}. The version presented here is slightly weaker than the
general result, but it will be sufficient for our needs. Note that a deterministic version of the
sewing lemma was given in  \cite{Max} and was instrumental for the reformulation of 
rough path theory \cite{Lyons} as exposed for example in \cite{FrizHairer}. 
A multidimensional analogue to the sewing lemma is given by the reconstruction theorem 
from the theory of regularity structures \cite[Thm~3.23]{Hairer}.

\begin{lemma}
[Stochastic Sewing Lemma]
\label{sewing-lemma}
Suppose that, for some $p\ge 2$, one has $A \in H_\eta^p \cap \bar H_{\bar \eta}^p$
with $\eta > {1\over 2}$ and $\bar \eta > 1$.
Then, for every $t > 0$, the limit in $L^p$
\begin{equ}[e:defIA]
I_{s,t}(A)\eqdef \lim_{|\CP|\to 0} \sum_{[u,v] \in \CP} A_{u,v}\;,
\end{equ}
with $\CP$ taking values in partitions of $[s,t]$,
exists
and there exists a constant $C$ depending only on $p$ and $\eta, \bar \eta$ such that
\begin{equs}
\|I_{s,t}(A)\|_{p} &\le C \bigl(\$A\$_{\bar \eta,p} |t-s|^{\bar \eta} + \|A\|_{\eta,p} |t-s|^{\eta}\bigr)\;,\\
\|\E (I_{s,t}(A) - A_{s,t}\,|\, \CF_s)\|_{p} &\le C \$A\$_{\bar \eta,p} |t-s|^{\bar \eta}\;.
\end{equs}
Furthermore $I(A)$ satisfies the identity $I_{s,u}(A) + I_{u,t}(A) = I_{s,t}(A)$ for
any $s \le u \le t$, so that there exists a stochastic process $I_t(A) = I_{0,t}(A)$ with 
$I_{s,t}(A) = I_t(A) - I_s(A)$.
If one furthermore has the bound $\|\E(A_{s,t}\,|\, \CF_s)\|_p \lesssim |t-s|^{\bar \eta}$, then
$I(A) \equiv 0$.
\end{lemma}

\begin{remark}
In the general case, the bound \eqref{e:supbound} is required for $\delta A$ only, but 
we will always have this stronger bound at our disposal.
Note also that \cite[Thm~2.1]{Khoa} requires joint continuity of $\E(A_{s,t}\,|\, \CF_s)$,
but this is only ever used to obtain \cite[Eq.~2.8]{Khoa} which we do not need.
\end{remark}

\begin{remark}
A simple special case is the classical result by Young \cite{Young}: for $f \in \CC^\alpha$ and
$g \in \CC^\beta$ with $\alpha + \beta > 1$, one has the bound
\begin{equ}[e:Young]
\Bigl|\int_s^t f_r\,dg_r - f_s (g_t - g_s)\Bigr|\lesssim |f|_\alpha |g|_\beta \,|t-s|^{\alpha+\beta}\;.
\end{equ} 
Strictly speaking, setting $A_{s,t} = f_s(g_t-g_s)$, this is only a special case of
Lemma~\ref{sewing-lemma} for $\beta > {1\over 2}$, but this is only due to the fact that,
as already mentioned, the formulation given here is slightly weaker than the one given in \cite{Khoa}.
\end{remark}
 
\subsection{A stochastic integral with respect to fBm}

For $\kappa, \gamma \in [0,1]$, we introduce a space $\CC^{-\kappa, \gamma}$
of distributions of order $-\kappa$ (in time) with values in the space of Hölder 
continuous functions of order $\gamma$ (in space). 
More precisely, an element $f \in \CC^{-\kappa,\gamma}$ is interpreted as the 
distributional derivative with respect to the first argument of a continuous function 
$\hat f\colon \R \times \R^d \to \R$,   such that $\hat f(0,x) = 0$ and
\begin{equs}[e:normkappa]
|\hat f(t,x) - \hat f(s,x)| &\le K |t-s|^{1-\kappa}\;,\\
|\hat f(t,x) - \hat f(s,x) - \hat f(t,y) + \hat f(s,y)| &\le K |t-s|^{1-\kappa}|x-y|^{\gamma}\;,
\end{equs}
uniformly over $|s-t| \le 1$ and $x,y \in \R^d$. 
%    Since $\hat f$ is a continuous function, the 
%    distribution $f$ can be tested against indicator functions, and interpreted  as $\int_0^t f(s,x)ds$.
Alternatively, one has
\begin{equs}
\sup_x |f(\fat,x)|_{-\kappa} \le K\;,\quad 
\sup_{x\neq y} {|f(\fat,x) - f(\fat,y)|_{-\kappa} \over |x-y|^{\gamma}} \le K\;. 
\end{equs}
We write $|f|_{-\kappa, \gamma}$ for the smallest possible choice of proportionality constant
$K$ in \eqref{e:normkappa}. In particular if $f$ is bounded and $f(
r, \cdot )$ uniformly $\gamma$-H\"older continuous (uniformly in $r$), then  $f\in\CC^{-\kappa, \gamma}$
for every $\kappa > 0$.

The following lemma is only used for the proof of Theorem~\ref{thm:indep}.
\begin{lemma}\label{lemma-convergence-deterministic}
For $\alpha,\kappa \in (0,1)$, the map
\begin{equ}
(f,x) \mapsto \big(t \mapsto f(t,x_t)\big)\;,
\end{equ}
extends to a continuous map from
$\CC^{-\kappa, \gamma} \times \CC^\alpha$ into $\CC^{-\kappa}$
provided that $\gamma \alpha > \kappa$. Furthermore,
one has the bound
\begin{equ}[e:boundComp]
|t \mapsto f(t,x_t)|_{-\kappa} \lesssim |f|_{-\kappa, \gamma} \bigl(1 +  |x|_\alpha^{\gamma}\, T^{\gamma \alpha}\bigr) \;,
\end{equ}
on any interval of length $T$.
\end{lemma}

\begin{proof}
This is an immediate consequence of the deterministic sewing lemma \cite{Max}:
Let $\Xi_{s,t}$ be a deterministic two parameter process with  
\begin{equ}
|\Xi_{st}| \le \hat K |t-s|^\eta\;,\qquad  |\delta \Xi_{sut}|\le  \hat K_\Lip |t-s|^{\bar \eta}\;,
\end{equ}
for some $\eta>0$ and $\bar \eta>1$. Then, for every $s < t \le T$, the limit 
$I_{s,t}(\Xi)\eqdef \lim_{|\CP|\to 0} \sum_{[u,v] \in \CP} \Xi_{u,v}$
exists along partitions of $[s,t]$ and one has 
\begin{equ}[e:boundIst]
|I_{s,t}(\Xi) - \Xi_{s,t}| \lesssim \hat K_\Lip |t-s|^{\bar \eta}\;.
\end{equ}
To make sense of the distribution $r\mapsto f(r,x_r)$, 
we need to be able to make sense of its integral over any interval $[s,t]$.
A good candidate for this is $I_t(\Xi)$, where
\begin{equ}
\Xi_{s,t} = \int_s^t f(r,x_s)\,dr = \hat f(t,x_s) - \hat f(s,x_s)\;.
\end{equ}
Writing $K = |f|_{-\kappa,\gamma}$, the bounds \eqref{e:normkappa} imply that
$|\Xi_{s,t}| \lesssim K |t-s|^{1-\kappa}$. On the other hand, we have
\begin{equs}
|\delta \Xi_{sut}| &= \big|\hat f(t,x_s) - \hat f(u,x_s) - \hat f(t,x_u) + \hat f(u,x_u)\big| \\
&\lesssim |f|_{-\kappa, \gamma} |t-u|^{1-\kappa} |s-u|^{\gamma \alpha} |x|_\alpha^{\gamma}\;,
\label{e:basicBound}
\end{equs}
so that the corresponding `integral' $I(\Xi)$ is well-defined since  $\gamma \alpha > \kappa$ by assumption. Furthermore, it follows from \eqref{e:boundIst} that
\begin{equ}
|I_{s,t}(\Xi)| \lesssim |t-s|^{1-\kappa} |f|_{-\kappa, \gamma}(1 + |x|_\alpha^{\gamma} |t-s|^{\gamma\alpha})\;,
\end{equ}
which does indeed show that the distributional derivative of $t \mapsto I_{0,t}(\Xi)$ belongs
to $\CC^{-\kappa}$. In the particular case where $f$ is actually a $\beta$-H\"older continuous function in its 
first argument, we have
\begin{equ}
|\Xi_{s,t} - f(s,x_s)(t-s)| \lesssim |t-s|^{1+\beta}\;,
\end{equ}
so that we do have $I_{s,t}(\Xi) = \int_s^t f(r,x_r)\,dr$ and therefore
${d\over dt}I_{0,t}(\Xi) = f(t,x_t)$ as required.
\end{proof}

We now introduce a space of stochastic processes that will be a natural candidate for containing
our solutions. Recall that $\CF_t$ denotes a filtration of the underlying probability space
as in Section~\ref{sec:basicEstimate}, namely it contains $\CG_s = \sigma(\{B_u -B_v\,:\, u,v \le s\})$
and is such that $\tilde B^s_{\fat}$ is independent of $\CF_s$ for every $s$.

\begin{definition}\label{space-CB}
For  $\alpha >0$ and  $p \ge 1$, let $\CB_{\alpha,p}$ denote the Banach space consisting of all $\CF_t$-adapted processes
$x_t$ such that $\delta x \in H_\alpha^p$, where $\delta x_{s,t} = x_t - x_s$. We also write
\begin{equ}[e:boundx]
\|x\|_{\alpha,p} = \sup_{s,t} |t-s| ^{-\alpha}\|x_t - x_s\|_{L^p} \;.
\end{equ}
(To be consistent with \eqref{e:supbound} we should really write $\|\delta x\|_{\alpha,p}$, but we drop the $\delta$
for the sake of conciseness.)
\end{definition}

Our aim is to lay the foundations for a solution theory of SDEs driven by fractional
Brownian motion with right hand sides determined by functions $f \colon \R \times \R^d \to L(\R^m, \R^d)$
such that the following holds.

\begin{lemma}\label{lem:mainTight}
Let $p\ge 2$ and $\alpha>\f 12$. Assume that  $x_{\fat} \in \CB_{\alpha,p}$, 
let $f \in \CC^{-\kappa,\gamma}$ (deterministic)
for some $\kappa, \gamma \ge 0$ such that $\eta = H-\kappa > {1\over 2}$ and 
$\bar \eta = H-\kappa + \gamma\alpha>1$, and define the two parameter stochastic process
\begin{equ}
A_{s,t} = \int_s^t f(r,x_s)\,dB_r\;,
\end{equ}
where the integral is interpreted as a conditional Wiener integral as constructed in Lemma~\ref{lemma-a}.
Then one has $A \in H_{\eta}^p \cap \bar H_{\bar \eta}^p$
and we take the resulting process as our definition of the stochastic integral against
$B$:
 \begin{equ}[e:defIntegral]
\int_s^t f(r,x_r)\,dB_r\eqdef  I_{s,t}(A)\;.
\end{equ}
This integral satisfies the bounds
 \begin{equs}\label{e:boundCA}
\Big \|&\int_s^t f(r,x_r)\, dB_r\Big\|_p \\
&\qquad\lesssim    |f|_{-\kappa, \gamma} \Big( 
|t-s|^{H-\kappa}+ \|x\|_{\alpha,p}^{\gamma} |t-s|^{\bar \eta}\Big) \;,\\
\Big\|&\E\Big( \int_s^t \bigl(f(r, x_r)-f(r,x _s)\bigr)\,dB_r\,\Big|\,\CF_s  \Big)\Big\|_{p} \label{e:boundCA2} \\
&\qquad\lesssim  |f|_{-\kappa, \gamma} \|x\|_{\alpha,p}^{\gamma} |t-s|^{\bar \eta}\;.\quad\qquad 
\end{equs}
uniformly over $s,t \in [0,T]$.
\end{lemma}

\begin{proof}
By Lemma~\ref{sewing-lemma}, it suffices to show that for $\kappa, \gamma$ as in the assumption, 
one has $A \in H_{\eta}^p \cap \bar H_{\bar\eta}^p$ with
 \minilab{w:wantedBoundA}
\begin{equs}[2]
 \eta&=H-\kappa\;, \qquad&  \|A\|_{ \eta,p} &\lesssim |f|_{-\kappa,\gamma}\;.\label{e:wantedBoundA2}\\
 \bar \eta&= H - \kappa + \gamma\alpha\;,  \qquad &   \$A\$_{ \bar\eta,p} &\lesssim |f|_{-\kappa, \gamma} \|x\|_{\alpha,p}^{\gamma} \;, \label{e:wantedBoundA1}
\end{equs}

Since $f(\fat,x_s)$ is $\CF_s$-measurable and $f\in \CC^{-\kappa, \gamma}$ with $\kappa\in [0, H-\f 	12]$, it follows from Lemma~\ref{lemma-a}
 that one has the bound
\begin{equ}[e:boundAst]
 \|A_{st}\|_p \lesssim \sup_{x\in \R^d} |f(\fat, x)|_{-\kappa} |t-s|^{H-\kappa}\;,
\end{equ}
where $C_p$ is a universal constant, thus yielding the bound \eqref{e:wantedBoundA2}.

We now bound $\delta A_{sut}$ for $u$ between $s$ and $t$.
Since
\begin{equ}
\delta A_{sut} = \int_u^t \bigl(f(r,x_s) - f(r,x_u)\bigr)\,dB_r\;,
\end{equ}
and since $s < u$, we are again in the setting
of Lemma~\ref{lemma-a}, which yields
\begin{equ}[e:boundDeltaA]
 \|\delta A_{sut}\|_p \lesssim \||f(\fat,x_s) - f(\fat,x_u)|_{-\kappa}\|_q |t-u|^{H-\kappa}\;.
\end{equ}
We then note that 
\begin{equ}
|f(\fat,x_s) - f(\fat,x_u)|_{-\kappa}
\le |f|_{-\kappa, \gamma} |x_s - x_u|^{\gamma}\;.
\end{equ}
Choosing $q = p/\gamma$ in \eqref{e:boundDeltaA},
we thus obtain
\begin{equ}
 \|\delta A_{sut}\|_p \lesssim |f|_{-\kappa,\gamma}  \|x\|_{\alpha,p}^{\gamma} |u-s|^{\alpha\gamma} |t-u|^{H-\kappa}\;,
\end{equ}
so that \eqref{e:wantedBoundA1} follows. Since $H-\kappa>\f 12$ and
 $H+\alpha \gamma-\kappa>1$  we can now apply 
Lemma~\ref{sewing-lemma} and immediately deduce
 $I_{s,t}(A)=\lim_{|\CP\to 0} \sum_{[u,v]\in \CP} A_{u,v}$ and the required bounds (\ref{e:boundCA}-\ref{e:boundCA2}).
\end{proof}

\begin{remark}
Since $f(t,x)$ is not assumed to be H\"older continuous in $t$, the integral defined in the theorem by sewing up the mixed integrals
(Riemann--Stieltjes integral with respect to the smooth $\bar B_t$ and the Wiener integral with respect to  $\tilde B_t$ with essentially `non-random' integrand)
cannot necessarily be interpreted as a Young integral.
\end{remark}

Finally, we note that in `nice' situations where the
 integral against $B$ also makes sense as a Young integral, the two integrals coincide.
The precise statement is as follows. 

\begin{lemma}\label{lem:sewing=Young}
Under the assumptions of Lemma~\ref{lem:mainTight} and assuming that $f$ is such that, for some $\delta$ with $\delta + H > 1$, one has
\begin{equ}
\sup_x \sup_{|t-s| \le 1} |t-s|^{-\delta} |f(t, x) -f(s, x)| < \infty\;,
\end{equ}
the integral given by \eqref{e:defIntegral} coincides with the usual Young integral.
\end{lemma}

\begin{proof}
By Lemma~\ref{lemma-a} with $\kappa=0$,  
and  $q > p \ge 2$:
\begin{equs}
\Big\| \int_{u}^{v}&\big( f(r, x_{u}) -f(u, x_{u})\big) \ dB_r \Big\|_p\\
&\le \Big\|   \sup_{ r\in [u, v]}  |f(r, x_{u}) -f(u, x_{u})| \Big\|_q\,  |v-u|^H\lesssim |v-u|^{H+\delta}\;.
\end{equs}
The final part of Lemma~\ref{sewing-lemma} then leads to the desired conclusion, namely that
$$ 
\lim_{|\CP|\to 0} \sum_{[u,v]\in \CP}\left(  \int_{u}^{v} f(r,x_u)\,dB_r-f(u, x_u)(B_v-B_u)\right)  = 0,
$$
in probability.
\end{proof}

\subsection{A semi-deterministic averaging result}
\label{sec:semidet}

In order to state the main theorem of this section, which is then going
to lead us to the proof of Theorem~\ref{thm:indep}, we introduce the space
$\CC^{\alpha,2}$ of functions that are $\alpha$-Hölder continuous in time, with 
values into the space $\BC^2$. With this notation, we then have the following. 

\begin{theorem}\label{weak-limits}
For $H > {1\over 2}$, let $\alpha, \kappa, \gamma > 0$ satisfy the 
assumptions of Lemma~\ref{lem:mainTight} and $\alpha < H-\kappa$.
Let furthermore $\zeta \in (\alpha,1]$ and
let $f_n, \bar f: \R_+\times \R^d \to L(\R^m, \R^d)$ 
be in $\CC^{\zeta,2}$ such that 
\begin{equ}
\lim_{n \to \infty} |f_n - \bar f|_{-\kappa,\gamma} = 0\;.
\end{equ}
Let $x^n$ and $x$ be the $\CC^\alpha$ solutions to the equations
\begin{equs}\label{limit-equations}
dx^n_t=f_n(t,x_t^n) \, dB_t, \qquad  dx_t=\bar f(x_t) \, dB_t\;,
\end{equs}
with $x^n_0 = x_0$ and the integrals interpreted pathwise in Young's sense.
Then, $x^n\to x$ in probability in $\CC^\alpha$. The same holds if 
the equations include a term with $dB$ replaced by $dt$.
\end{theorem}

\begin{proof}
The fact that \eqref{limit-equations} admits unique solutions in $\CC^\alpha$ 
for every realisation of $B \in \CC^\beta$ with $\beta > \alpha$ and $\alpha + \beta > 1$
is standard. (Combine Lemma~\ref{lem:Lip} with \cite{Young} to show that the Picard
iteration is contracting in $\CC^\alpha$ with fixed initial value.)
 
Let us first obtain bounds on $x^n$ that are uniform in $n$. 
For any $\f12 < \alpha < H-\kappa$ satisfying $\alpha \gamma> 1-H+\kappa$, we can apply  bound \eqref{e:boundCA} of Lemma~\ref{lem:mainTight}
so that, over any interval $[0,T]$, we obtain the bound
\begin{equ}
\|x^n\|_{\alpha,p}
\lesssim T^{H-\alpha-\kappa}\big(1 + T^{\gamma\alpha} \|x^n\|_{\alpha,p}^{\gamma}\bigr)\;,
\end{equ}
which immediately implies that $\|x^n\|_{\alpha,p}\le 1$, uniformly over $n$, provided
that we choose a sufficiently short time interval. This bound can be iterated and therefore
yields an order one a priori bound on $\|x^n\|_{\alpha,p}$ over any fixed time interval.

We then note that we can write
\begin{equ}
x_t = Z_t + \int_0^t \bar f(x_s)\,dB_s\;,\qquad
x^n_t = Z^n_t + \int_0^t \bar f(x^n_s)\,dB_s\;,
\end{equ}
with
\begin{equ}
Z_t = x_0 \;,\qquad Z^n_t = x_0 + \int_0^t \big(f_n(s,x^n_s) - \bar  f(x^n_s)\big)\,dB_s\;.
\end{equ}
It now follows again from (\ref{e:boundCA}) in Lemma~\ref{lem:mainTight} that, over any fixed time interval,
one has the bound
\begin{equ}
\|Z^n - Z\|_{H-\kappa,p} \lesssim |f_n-\bar f|_{-\kappa,\gamma} (1+\|x^n\|_{\alpha,p}^{\gamma})\;.
\end{equ}
Note now that by Kolmogorov's continuity theorem, we have for any $\delta, \zeta > 0$ the inclusions
\begin{equ}[e:Kolmogorov]
L^p(\Omega, \CC^\zeta) \subset \CB_{\zeta,p} \subset L^p(\Omega, \CC^{\zeta - {1\over p}-\delta})
\end{equ}
so that, choosing $p$ large enough, we conclude that 
$|Z^n - Z|_\alpha \to 0$
 in $L_p$, for any $p\ge 2$,  as $n \to \infty$. 
The claim now follows from Lemma~\ref{general-convergence}. 
\end{proof}

The type of application of this theorem that we have in mind is that when $f_n$ is 
for example given by
\begin{equ}[e:speedF]
f_n(t,x) = F(x,y_{nt})\;,
\end{equ}
for some smooth function $F$ and for a stationary stochastic process $y_t$ 
that is independent of the driving noise $B$. (This is so that $f_n$ can be considered deterministic.)
To give a more concrete setting, given any two random variables $X$ and $Y$, we
can measure their degree of independence $\alpha(X,Y)$ (also called the 
`strong mixing coefficient') by
\begin{equ}
\alpha(X,Y) = \sup\big\{\P (A \cap B) -  \P(A)\, \P(B)\,:\, A\in \sigma(X), B\in \sigma(Y)\}\;.
\end{equ}
Note that if $F$ and $G$ are two bounded centred functions, then 
\begin{equ}[e:covarBound]
\bigl|\E F(X) G(Y)\bigr| \le 4\alpha(X,Y) |F|_\infty|G|_\infty \;,
\end{equ}
see \cite{Ibragimov}.
The following proposition is then crucial.

\begin{lemma}\label{le:crucial}
Let $\CY$ be a Polish space and let $(y_t)_{t\in \R}$ be a stationary $\CY$-valued 
stochastic process such that 
$\alpha(y_0,y_t) \lesssim t^{-\delta}$
for some $\delta > 0$. Let $F \colon \R^d \times \CY \to \R$ be
a measurable function, $\CC^2$ in the first variable, such that 
\begin{equ}
|F(x,y)| \le K\;,\qquad |F(x,y) - F(z,y)| \le K|x-z|\;,
\end{equ}
uniformly over $y \in \CY$ and $x,z \in \R^d$. Assume for simplicity that outside of a compact set $F$ is periodic in its first argument.

Then, for every $\kappa > 0$, every $\gamma < 1$ and every $p \ge 1$, the sequence $f_n$ defined as in
\eqref{e:speedF} is such that $|f_n - \bar f|_{-\kappa,\gamma} \to 0$ in $L_p$ as $n \to \infty$, with
$\bar f(x) = \int F(x,y)\,\mu(dy)$, where $\mu$ denotes the law of $y_t$ for any fixed $t$.
\end{lemma}

\begin{proof}
 Since $F$ is bounded measurable and $f_n(t,x)=F(x, y_{nt}) \in \CC^{0,1}$, we 
 note that $\bar f $ is bounded Lipschitz continuous.
Replacing $f_n$ by $(f_n-\bar f)/K$, we can assume without loss of generality that $K=1$ and 
$\bar f = 0$. 
Making use of \eqref{e:covarBound}, we then have the bound
\begin{equs}
\E \Big(\int_s^t f_n(r,x)\,dr\Big)^2
&= n^{-2}\E \int_{ns}^{nt}\int_{ns}^{nt}F(x,y_{r})F(x,y_{\bar r})\,dr\,d\bar r \\
&\le|F(x, \cdot)|_\infty^2\,  4 n^{-2}\int_{ns}^{nt}\int_{ns}^{nt}\alpha(y_r,y_{\bar r})\,dr\,d\bar r \\
&\lesssim |F(x, \cdot)|_\infty^2 \, n^{-2}\int_{ns}^{nt}\int_{ns}^{nt} |r-\bar r|^{-\delta}\,dr\,d\bar r \\
&\lesssim |F(x, \cdot)|_\infty^2 \,n^{-2} |nt-ns|^{2-\delta} \lesssim n^{-\delta} |t-s|^{2-\delta}\;.
\end{equs}
On the other hand, we have the trivial bound $\big|\int_s^t f_n(r,x)\,dr\big| \le |t-s|$,
so that for any $p\ge 2$,
\begin{equ}
\Big\|\int_s^t f_n(r,x)\,dr\Big\|_p \lesssim n^{-\f \delta p} |t-s|^{1-\f\delta p}\;.
\end{equ}
Replacing $f_n(r,x)$ by $f_n(r,x)- f_n(r,z)$, we similarly obtain
\begin{equ}
\Big\|\int_s^t\bigl( f_n(r,x) -  f_n(r,z)\bigr)\,dr\Big\|_p \lesssim n^{-\f \delta p} |x-z|  |t-s|^{1-\f\delta p}\;.
\end{equ}
Applying Kolmogorov's continuity criterion, we obtain over any finite time
interval the bound
\begin{equ}
\Big\|\Big|\int_0^{\fat} \bigl( f_n(r,x) -  f_n(r,z)\bigr)\,dr\Big|_{1-\kappa} \Big\|_p
\lesssim n^{-\f \delta p} |x-z|\;,
\end{equ}
provided that we choose $p$ large enough so that
$\f\delta p + \f1p < \kappa$. In other words, we have the bound
\begin{equ}
\big\|\big|f_n(\fat,x) -  f_n(\fat,z)\big|_{-\kappa} \big\|_p
\lesssim n^{-\f \delta p} |x-z|\;.
\end{equ}
This allows us to apply Kolmogorov's criterion a second time, this time in the 
spatial variable, showing that for any compact set $\CK$, we have
\begin{equ}
\Big\|\sup_{x\neq z\atop x,z \in \CK}{\big|f_n(\fat,x) -  f_n(\fat,z)\big|_{-\kappa} \over |x-z|^{\gamma}}
\Big\|_p \lesssim n^{-\f \delta p} \;,
\end{equ}
provided that $p$ is such that $\f dp < 1-\gamma$,
where the proportionality constant in front of $n^{-\f \delta p}$ depends on the compact set $\CK$.
Since $F$ is furthermore assumed to be periodic outside of a compact set, this bound
 extends to the whole space, proving the claim that $|f_n-\bar f|_{\kappa, \gamma}\to 0$
 in $L_p$.
 \end{proof}
 
\begin{proof}[of Theorem~\ref{thm:indep}]
As above, we define
\begin{equ}
f_n(t,x) = f(x,y_{nt})\;,
\end{equ}
for $f$ as in the statement of the theorem. Since 
$\bar f$ is Lipschitz continuous and $\CC^2$, the equation $\dot x_t=\bar f(x_t)\,dB_t$ has a unique solution.
Assume now that $x^n$ is a $\CC^\alpha $ solution to $dx_t^n=f_n(t,x_t^n) \,dB_t$. 
Note that since  $f_n\in \CC^{0,1}$, 
 $t\mapsto f_n (t, x_t^n)\in \CC^{-\kappa}$ for any $\kappa<\alpha$  (by Lemma \ref{lemma-convergence-deterministic})
and the integral $\int_0^tf_n(s, x_s^n) \,dB_s$
makes sense by Lemma~\ref{lem:mainTight}, so the notion of what constitutes a solution is unambiguous.

We now modify $f$ outside of the ball $B_R$ of radius $R$ centred at $0$ so that the resulting function $f_R$ 
is periodic in its first argument 
and satisfies the conditions of Lemma~\ref{le:crucial}.
Write $f_n^R(t,x)\eqdef f_R(x,y_{tn})$
and denote by $x_t^{n,R}$ and $x^{R}_t$ the respective solutions to 
 $d x_t=f_R(x_t, y_{nt})\,dB_t$ and $dx_t=\bar  f_R (x_t)\,dB_t$. 
 By Theorem~\ref{weak-limits}, as $n\to \infty$, 
$x_t^{n,R}$ converges to $x_t^R$ in probability. The reason why we are able to 
apply this result is that, even though the functions $f_n^R$ are not deterministic,
they are independent of $B$.
Furthermore, one has $x_t^{n,R}=x_t^n$
and $x_t^R=x_t$ before they exit $B_R$ so that, sending $R \to \infty$, we conclude that
$x_t^n\to x_t$ in probability, as required.
\end{proof}

To conclude this section, we give a deterministic example of averaging. Fix
a function 
\begin{equs}
F \colon \R \times \R^n \times \R_+ &\to L(\R^m, \R^d) \\
(t,x,\tau) &\mapsto F_\tau(t,x)
\end{equs}
with the property that, for
any fixed $\tau \in \R_+$, the function $F_\tau$ is of class $\BC^2$ and is 
periodic with period $\tau$ in its first argument. We furthermore
assume that, for some positive Radon measure $\mu$ on $\R_+$ and some $ \kappa > 0$, one has
\begin{equ}[e:boundF]
\int_0^\infty  |F_\tau|_{\BC^2} (1+\tau^{ \kappa})\,\mu(d\tau) < \infty\;.
\end{equ}
We then set
\begin{equ}
f_n(t,x) = \int_0^\infty F_\tau(nt,x)\,\mu(d\tau)\;,\qquad
\bar f(x) = \int_0^\infty \f1\tau \int_0^\tau F_\tau(t,x)\,dt \,\mu(d\tau)\;.
\end{equ}

\begin{remark}
A special case is when $\mu$ is atomic, which corresponds to the case when $f_n$ is 
a sum of periodic functions.  
\end{remark}

\begin{proposition}
In the above setting, if we set
$$
d x^{(n)}_t= f_n(t,x^{(n)}_t) \, dB_t\;,\qquad 
d \bar x_t=\bar  f(\bar x_t) \, dB_t\;,
$$
then $x^{(n)}$ converges in probability to $\bar x$.
\end{proposition}

 \begin{proof}
We have $f_n, \bar f \in \BC^2$ as an immediate consequence of their definitions and \eqref{e:boundF}.
Note also that if $g_\eps$ is periodic with period $\eps$ and averages to $0$, then one has the bound
\begin{equ}
\Bigl|\int_s^t g_\eps(r)\,dr \Bigr| \le |g_\eps|_\infty \bigl(|t-s| \wedge \eps\bigr)\;.
\end{equ}
As a consequence, one has
\begin{equs}
\Big|\int_s^t \bigl(f_n&(r,x) - \bar f(x)\bigr)\,dr\Big|
\lesssim \int_0^\infty |F_\tau(\cdot, x)|_{\infty} \bigl(|t-s| \wedge n^{-1}\tau\bigr)\,\mu(d\tau) \\ 
&\lesssim |t-s|^{1-\kappa} \int_0^\infty |F_\tau(\cdot, x)|_{\infty}  n^{-\kappa}\tau^\kappa\,\mu(d\tau) 
\lesssim n^{-\kappa}|t-s|^{1-\kappa}\;.
\end{equs}
Since one similarly has the bound
\begin{equ}
\Big|\int_s^t \bigl(f_n(r,x) - f_n(r,y) - \bar f(x) + \bar f(y)\bigr)\,dr\Big|
\lesssim n^{-\kappa}|x-y|\,|t-s|^{1-\kappa}\;,
\end{equ}
it follows that 
$
|f_n-\bar f|_{-\kappa,1} \lesssim n^{-\kappa}
$,
and we conclude by Theorem \ref{weak-limits}.
 \end{proof}
% Other interesting cases are:  $F(\cdot, x) $ periodic  with period depending on $x$ and $f_n(t,x) =F(nt,x)$.
 
\section{Averaging with feedback}
\label{sec:averaging}

We now turn to the main result of this article, where we allow for feedback from the slow
dynamic into the fast dynamic. The trade-off is that our averaging result is not as general
as Theorem~\ref{weak-limits},  as we require that the fast dynamic is Markovian.

\subsection{A class of slow\slash fast processes}\label{thm:B}

Fix a smooth compact manifold $\CY$ for the fast variable and consider the slow\slash fast system 
\minilab{e:equation}
\begin{equs}
dx_t^\eps &= f(x_t^\eps,y_t^\eps)\,dB_t + g(x_t^\eps,y_t^\eps)\,dt\;, \label{e:equationx}\\
dy_t^\eps &= {1\over \eps} V_0(x_t^\eps,y_t^\eps)\,dt + {1\over \sqrt \eps}V(x_t^\eps,y_t^\eps)\circ d\hat W_t\;, 
\label{e:equationy}\\
x_0^\eps&=x_0\in \R^d, \qquad y_0^\eps=y_0\in \CY\;,
\end{equs}
where $B$ is an $m$-dimensional fractional Brownian motion with Hurst parameter $H > {1\over 2}$
and $\hat W$  an $\hat m$-dimensional
standard Wiener process  independent  of $B$. Also, $f: \R^d\times \CY \to  L(\R^{m},\R^d)$ and 
$g:\R^d\times \CY \to  \R^d$.  We use the shorthand 
$$V(x_t^\eps,y_t^\eps)\circ d\hat W_t\eqdef \sum_{k=1}^{\hat m}V_k(x_t^\eps,y_t^\eps)\circ d\hat W_t^k$$ for
vector fields $V_i(x,\fat)$   on $\CY$. Similarly $f(x_t^\eps,y_t^\eps)\,dB_t\eqdef \sum_{k=1}^m f_k(x_t^\eps,y_t^\eps)\,dB_t^k$.
We fix a Riemannian metric on $\CY$ and furthermore assume that the following holds.
%Let $\tilde V_0=V_0+\f 12 \sum \nabla _{V_i}V_i$ denotes the effective drift.

\begin{assumption}\label{ass:main}
The drift vector field $g$ is uniformly bounded and globally Lipschitz continuous. Also   $f, V_0 \in \BC^2$ and $V_k \in \BC^3$ for $k > 0$.
Furthermore, there exists $\lambda > 0$ such that, for 
all $x \in \R^d$, $y \in \CY$ and $v \in T_y\CY$, one has $\sum_{j>0} \scal{v,V_j(x,y)}^2 \ge \lambda |v|^2$.
\end{assumption}
%For every $x_t$ adapted in $\CB_{\alpha,p}$, not necessarily a solution to (\ref{e:equationx}), there is a unique 
%solution to (\ref{e:equationx}). Furthermore, the solution has a continuous version. Similarly
%for every $y_t$ that is H\"older continuous

Solutions will be interpreted as follows. We fix a realisation of the fractional Brownian motion in $\CC^\beta$ 
for some $\beta > {1\over 2}$ and we will look for solutions 
that are  H\"older continuous  of order $\alpha $ for some $\alpha< \f 12$ with $\alpha + \beta > 1$, 
so that integration with respect to the realisation of fractional Brownian motion can (and will)
be interpreted as a Young integral. 
%Regarding the equation for $y$, one can always view
%$\CY$ as a smooth submanifold of $\R^d$ for $d$ large enough and extend the $V_i$ to $\BC^2$ vector fields
%on all of $\R^d$. One can then write the equation for $y$ as an It\^o equation, the right hand side of
%which makes sense for any adapted process $x$.
We will see in Theorem~\ref{thm:genSol} below  that
(\ref{e:equation}) is well posed and admits solutions in $\CB_{\alpha,p}$ for any $\alpha<\f 12$. Indeed,
 by the consideration in Remark~\ref{remark-existecne-y},  it is sufficient to consider 
 the case where $\CY$ is a Euclidean space.
A posteriori, one can easily show that the slow variables actually 
satisfy $x_{\fat}^\epsilon  \in \CB_{\beta,p}$ for any $\beta<H$, but this will not be needed.

For any fixed $(x, y) \in \R^d \times \CY$ and any time $s\in \R$, consider the SDE
\begin{equs}[e:xfrozen]
dY_{s,t} = {1\over \eps} V_0(x,Y_{s,t})\,dt + {1\over \sqrt \eps} V(x,Y_{s,t})\circ d\hat W_t\;,\qquad
Y_{s,s}=y\;,
\end{equs}
with $t \ge s$. We write  $Y_{s,t} = \bar \Phi^x_{s,t}(y)$ for the solution flow associated to this equation, 
which exists for all time and is unique under our assumptions. 
The superscript denotes the frozen variable $x$ and we have refrained from adding also $\eps$ to the notation.

Write now
$\CP_t^x$ for the Markov transition semigroup on $\CY$ with generator
\begin{equ}
\CL^x = {1\over \eps} V_0(x,\fat) + {1\over 2\eps} \sum_{i=1}^{\hat m} (V_i(x,\fat))^2\;,
\end{equ} 
(vector fields are identified with first order differential operators in the usual way),
so for any points $x\in \R^d$, $y\in \CY$, and bounded measurable $F: \CY\to \R$, 
$\CP_t^xF(y)=\E F(\bar \Phi^x_{s,t}(y))$.
Note that since $\CY$ is compact and the diffusion for $y$ is uniformly elliptic for any $x\in \R^d$, 
the solution $\bar \Phi^x_{s,t}(y)$ with generator $\CP_t^x$ admits a unique invariant 
probability measure $\mu^x$ on $\CY$. 

\begin{remark}
The semigroup $\CP_t^x$ depends on $\eps$, but in a trivial way, i.e.\ only though a time change.
In particular, the family $\mu^x$ of invariant measures does not depend on $\eps$.
\end{remark}

Writing $\bar f(x) = \int_\CY f(x,y)\,\mu^x(dy)$ and similarly for $\bar g$, the following is our main result, the proof of which will be given in Section~\ref{sec:proofmain}.

\begin{theorem} \label{thm:main}
Assume Assumption~\ref{ass:main}.  Let $B_t$ be a fBm with Hurst parameter $H>\f 12$ and let $\hat W_t$ be an independent Brownian motion. Then over any finite time interval $[0,T]$
and for any $\beta<H$, the solution $x_t^{\eps}$ to \eqref{e:equation} converges, 
in probability in $\CC^\beta$ as $\epsilon\to 0$,  to the unique limit $\bar x_t$ solving
\begin{equ}
d\bar x_t = \bar f(\bar x_t)\,dB_t + \bar g(\bar x_t)\,dt\;, \qquad \bar x_0=x_0\;.
\end{equ}
Furthermore, there exists an exponent $\kappa > 0$ (depending on $\beta$) such that 
\begin{equ}
\lim_{\eps \to 0} \P \bigl(| x_t^\epsilon-\bar x_t|_{\beta} > \eps^\kappa\bigr) = 0 \;.
\end{equ}
\end{theorem}

\begin{remark}
The lengthy part of the proof is to obtain a priori estimates on the slow variables $\{x_{\fat}^\eps\}$ that are uniform in $\epsilon$.  
The Young bound is of course useless
since the H\"older norm of $r \mapsto y_r^\eps$ diverges as $\eps \to 0$.
 It is more advantageous to use the sewing technique.
Since  $y_r^\eps$ contains noise not independent of the increments of $B_r$, 
\eqref{e:defIntegral} cannot be directly applied to $\int_s^t f(x_s^\eps, y_r^\eps)\,dB_r$.
Instead we break the interaction between the slow and the fast variables by 
replacing $y_r^\epsilon$ by  $Y_{s,r}\eqdef \bar \Phi_{s,r}^{x_s}(y_s)$ and
exploit the fact that, since $x_s^\eps$ is left frozen at its value at time $s$ and $Y_{s,t}$ is driven by 
$\hat W$, the integrand is independent of $\tilde B_r^s$. 
In order to use these for obtaining estimates we first prove that for any fixed $\eps> 0$, our notion of
integral, used for the purpose of estimation,  does still coincide with usual Young integration, see \S\ref{sec:mixed=Young}.

\end{remark}

We  also consider equation (\ref{e:equationy}) separately, with $x_t$ a given $\CF_t$ stochastic process
 not necessarily the solution to (\ref{e:equationx}). More precisely we consider
\begin{equ}\label{fast1}
dy_t^\eps = {1\over \eps} V_0(x_t,y_t^\eps)\,dt + {1\over \sqrt \eps} V(x_t,y_t^\eps)\circ d\hat W_t\;,
\end{equ}
for $x$ any given sufficiently regular $\CF_t$-adapted stochastic process.
Its solution flow will be denoted by  $\Phi^{x}_{s,t}$, namely $ \Phi^x_{s,t}(y,\omega)$ is the solution to the equation
at time $t \ge s$ with $ \Phi^x_{s,s}(y,\omega) = y$. For its existence see Remark~\ref{remark-existecne-y}. 
The chance element $\omega$ in the flow is  often omitted for simplicity
and the superscript denotes the dependence on the auxiliary process $x_t$.
%Again, we first fix a realisation of the fractional Brownian motion, then interpret the integrals as Young integrals. 

Given an adapted $\R^d$-valued stochastic processes $x_t$
and an initial condition $y_0$, we will henceforth use the symbol $Y_{s,t}$ in order to denote the 
process
\begin{equ}[e:defYst]
Y_{s,t}\eqdef \bar \Phi^{x_s}_{s,t}\big(\Phi_{0,s}^x(y_0)\big)\;,
\end{equ}
namely $Y_{s,\fat}$ is the solution to 
 (\ref{e:xfrozen}) with frozen parameter $x=x_s$ and with initial condition $y=y_s$,
with $y_s$ itself given by the solution to \eqref{fast1}.

\begin{remark}
\label{remark-existecne-y}
Firstly we assume that (\ref{fast1}) is defined on the Euclidean space $\R^{d'}$.  Suppose $V_0 \in C^1$ and $V_1, \dots, V_m$ are $C^2$.  
Set $\tilde V(t, y, \omega)=V(x_t(\omega), y)$,
the randomness in $x_t$ is independent of that in $\hat W_t -\hat W_s$, so there exists a unique global solution to the SDE
$dy_t= \tilde V(t, y_t, \omega) \circ d\hat W_t+\tilde V_0(t, y_t, \omega)dt$,
which follows from the fixed point argument and the condition $|V(x,y)-V(x,y')|\le K d(y,y')$. 
On a compact manifold the global existence is trivial.
Let $(U_i, \phi_i)$ be an atlas of charts in $M$ with the property that  for every $i$, 
$\phi_i(U_i) $ contains the centred ball $B(3r)$ of radius $3r$ and the pre-image of $V_i\eqdef \phi_i(B(r))$ covers the manifold. 
Consider the SDE $dy_t=(\phi_i)_* \tilde V(t, y_t) \circ d\hat W_t+(\phi_i)_* \tilde V_0(t, y_t)dt$.
Since the the vector fields  $(\phi_i)_*(V_i)$ have uniform bounds on $B(3r)$,  there are uniform  estimates, in $i$, on the exit time of
$y_t$ from $V_i$ to $U_i$.\end{remark}

%\begin{remark}
% We comment on the existence of solutions to (\ref{e:equation}).
% By the consideration in Remark~\ref{remark-existecne-y},
% we may simply assume that $\CY$ is a linear space.
%Then the existence and uniqueness of a strong   strong solution of (\ref{e:equationx-1},~\ref{e:equationy-1}) in the class of $\alpha$-H\"older continuous processes
%with $\alpha<\min \{\f 12, \f \delta 2\}$ processes with its H\"older norm in $L_2\cap L_\infty$.
%
%\end{remark}

\subsection{Estimates for SDEs with mixed Young and It\^o integration}
\label{sec:non-uniform-bounds}
In this section, we show that SDEs driven by a fractional Brownian motion with $H > {1\over 2}$ and
a Wiener process do admit solutions in $\CB_{\alpha, p}$ for arbitrary $p$.  This is similar to the
results obtained in  \cite{Kubilius,Guerra-Nualart}, but since our spaces are slightly
different, our result doesn't appear to follow immediately from theirs. 
The first estimate below is a semi path by path result: we fixed a realisation of the fBm, then consider the It\^o integral. 

\begin{theorem}\label{thm:genSol}
  Let $b\in \CC^\beta$ where $\beta> \f 12$ and consider the equation in $\R^d$
  \begin{equ}[e:generalSDE]
 dz_t=F(z_t) \,db_t+ \sigma(z_t) \,d\hat W_t+ G(z_t)\, dt\;,
 \end{equ}
for some $F \in \BC^2$ and $\sigma, G \in \BC^1$.
Here, the first term on the right hand side is a Young integral while the second term is an Itô integral. 
Given a time interval $[0, T]$
and numbers $\alpha > 0$ and $p\ge 1$ with $1-\beta < \alpha<\f 12 $ (allowing $1-\beta < \alpha<\beta $ if $\sigma$ vanishes identically),
 there exists a unique solution in $\CB_{\alpha, p}$.
 Furthermore,
\begin{equ}[e:wantedBoundGeneral]
 \|z\|_{\alpha, p}  \lesssim   |b|_\beta^{\f1\beta}+1\;.
\end{equ}
 %$\alpha+\beta-\f 1p>1$,
  \end{theorem}
\begin{proof}
Note that the conclusion is more restrictive for larger values of $p$ so we can 
choose $p$ sufficiently large so that $\alpha > \f1p$ and $\alpha + \beta > 1 + \f1p$.
Let $z \in \CB_{\alpha,p}([0,\delta])$ and recall that, by Kolmogorov's continuity theorem, 
we have $\CB_{\alpha, p} \subset L_p(\Omega,\CC^\gamma)$ for any $\gamma<\alpha -\f 1p$ and,
for any fixed $\kappa > 0$, one has
\begin{equ}[e:boundKolmogorov]
\||z|_\gamma\|_p \lesssim \delta^{\alpha - \f1p -\kappa - \gamma} \|z\|_{\alpha,p}\;.
\end{equ}
(This is because, on an interval of length $\delta$,
$|z|_\gamma \le \delta^{\bar \gamma - \gamma}|z|_{\bar \gamma}$ for any $0 < \gamma \le \bar \gamma \le 1$.)
Let us define, for $z \in \CB_{\alpha,p}([0,\delta])$,
\begin{equs}
\Psi(z)&=z_0+ \int_0^{\fat} F(z_r) \,db_r+ \int_0^{\fat} \sigma(z_r)\, d\hat W_r+ \int_0^{\fat} G(z_r)\, dr\\
 &=z_0+\Psi_1(z)+\Psi_2(z)+\Psi_3(z)\;,
\end{equs}
and show that
for a sufficiently small value of $\delta$, $\Psi$ maps the ball of radius $1$ in 
$\CB_{\alpha, p}([0,\delta])$ centred around $z_0$ to itself.

Choosing $\gamma$ such that $\gamma + \beta > 1$ (which is always possible by taking $p$ large enough
and $\kappa$ small enough) and
using Young's bound \eqref{e:Young}, we obtain the estimate 
\begin{equs}\label{estimate-growth}
|\Psi_1(z)|_\alpha 
 &\lesssim    |F(z)|_\gamma |b|_\beta \delta^{\beta+\gamma-\alpha}+ |F|_\infty\, |b|_\beta  \delta^{\beta-\alpha}\\
&\le  |F|_\Lip  |z|_\gamma  |b|_\beta  \delta^{\beta+\gamma-\alpha}  +|F|_\infty  |b|_\beta  \delta^{\beta-\alpha}\;,
\end{equs}
where the H\"older semi-norm $|z|_\gamma$ is really the H\"older semi-norm of $z\restr [0, \delta]$, 
while the H\"older semi-norm $|b|_\beta$ is considered on the full interval $[0,T]$.
Combining this with \eqref{e:boundKolmogorov} and using the fact that 
$\|\fat\|_{\alpha,p}\le \||\fat|_\alpha\|_p$, we obtain the a priori bound  
\begin{equ}
\|\Psi_1(z)\|_{\alpha,p} 
\lesssim 
 \|z\|_{\alpha,p}  |b|_\beta  \delta^{\beta - \f1p - \kappa}  +  |b|_\beta  \delta^{\beta-\alpha}\;,
\end{equ}
where we used that $\alpha + \beta -\f 1p > 1$.
As a consequence of the Burkholder--Davis--Gundy inequality, we immediately obtain the bound
$$\|\Psi_2(z)\|_{\alpha, p}  \lesssim |\sigma |_\infty \delta^{\f 12-\alpha}\;,
$$
provided that $\alpha < \f12$, $p$ is large enough, and $\kappa$ is small enough.
This is the only place where we require that $\alpha<\f 12$, which is due to the regularity of the 
Wiener process. If the Wiener process is absent, we can choose any $\alpha\in (0, \beta)$.
Finally, it is trivial that $$\|\Psi_3(z)\|_{\alpha, p}  \lesssim |G |_\infty \delta^{1-\alpha}.$$

This shows that, assuming that $\Psi$ does admit a fixed point in $\CB_{\alpha,p}$, this fixed point necessarily satisfies 
the bound
\begin{equs}
\|z\|_{\alpha, p}\lesssim  |b|_\beta \delta^{\beta -\alpha}
+ \delta^{\f 12-\alpha} + \delta^{1-\alpha}\;,
\end{equs}
provided that $\delta^\beta |b|_\beta \le c$ for a sufficiently small constant $c > 0$
(depending on $F$). Since this bound is independent of the initial condition, it can be iterated
and necessarily holds on any interval of size $\delta$.
It is then straightforward to verify from the definitions that, on the interval $[0,T]$,
one has
\begin{equ}
\|z\|_{\alpha,p} \lesssim \delta^{\alpha-1}  \sup_{s \in [0,T-\delta]} \|z \restr [s,s+\delta]\|_{\alpha,p}\;,
\end{equ}
which implies that, fixing $\delta$ with $\delta^\beta |b|_\beta = \f 12$,
\begin{equ}
\|z\|_{\alpha, p}\lesssim  |b|_\beta \delta^{\beta -1}
+ \delta^{-\f 12} + 1 \lesssim 1 + |b|_\beta^{\f 1\beta} \;,
\end{equ}
as claimed.

To show that such a solution exists and is unique, we note that by 
\cite[Thm~2.2]{Guerra-Nualart}, there exists a unique adapted process $z \in L^2(\Omega, \CC^\alpha)$ 
solving \eqref{e:generalSDE}. To show that it furthermore satisfies the stronger bound
\eqref{e:wantedBoundGeneral}, write $\tau$ for the stopping time given by
\begin{equ}
\tau_M = T\wedge \inf \{t \in (0,T]\,:\, |z\restr [0,t]|_\alpha \ge M\} \;.
\end{equ}
The process $z^M$ obtained by stopping $z$ at time $\tau_M$ then belongs to $\CB_{\alpha,p}$ 
and the same calculation as above shows that it satisfies the bound \eqref{e:wantedBoundGeneral}.
Since one has $\lim_{M \to \infty} \P(z^M = z) = 1$, the claim  follows at once.
\end{proof}

Since the fractional Brownian motion has moments of all order, we may take average over all fractional Brownian paths and obtain the corollary below. 
\begin{corollary}
Suppose that $B$ and $\hat W$ are independent and let $F$, $G$ and $\sigma$ be as in Theorem~\ref{thm:genSol}.
 Let $\alpha \in (\f 1 p, \f 12)$ be such that $\alpha+H>1$.
 Then for any initial value and any interval $[0, T]$, 
 there exists a unique solution in $\CB_{\alpha,p}$ to
  \begin{equ}
 dz_t=F(z_t) \,dB_t+ \sigma(z_t) \,d\hat W_t+ G(z_t) \,dt\;.
 \end{equ}
 Furthermore, $\|z\|_{\alpha, p}  \lesssim  1+ \||B|^{\f1\beta}_\beta\|_p$. (If $\sigma$ vanishes, we may take  $\alpha \in (\f 1p, H)$.)
 \end{corollary}

\subsection{Stochastic equals Young}
\label{sec:mixed=Young}

We will make use of the following fact.

\begin{lemma}\label{lem:boundFunny}
Let $\alpha \in (0,1)$ and let $F$ and $G$ be two positive functions such that
\begin{equ}
F(t)\le \int_0^t F^{1-\alpha}(s) \,G(s)\,ds\;.
\end{equ}
Then, one has the bound
$ F^\alpha(t)\le  \alpha\int_0^t \,G(s)\,ds$.
\end{lemma}

\begin{proof}
Let  $\hat F(t) = \int_0^t F^{1-\alpha}(s) \,G(s)\,ds$, so that $\hat F$
is a continuous increasing function and we have the bound
\begin{equ}
{d\over dt}\hat F^\alpha(t) = \alpha \hat F^{\alpha-1} (t) F^{1-\alpha}(t) \,G(t)
\le \alpha G(t)\;,
\end{equ}
since $ F^{1-\alpha} \le \hat F^{1-\alpha}$. The claim then follows since $F^\alpha \le \hat F^\alpha$.
\end{proof}

Recalling $Y_{s,t}$ given by (\ref{e:defYst})
(with $x_t$ arbitrary, not necessarily solution to \eqref{e:equationx}), standard methods
for estimating its deviation  from $y_t$ on the time scale of $[s,t]$ 
 blow up exponentially fast as $\epsilon \to 0$.
(For longer times, we will use the smoothing 
properties of the semi-group). Recall that $y_u=\Phi^x_{0,u}(y_0)=\Phi^x_{s,u} (y_s)$ and $Y_{s,u}=\bar \Phi_{s,u}^{x_s}(y_s)$, and denote by $\rho$ the Riemannian distance on $\CY$.

\begin{lemma}\label{lemma:bounddiffy}
 Suppose $x_{\fat} \in \CB_{\alpha, p}$ where $\alpha<\f 12$ and $\alpha+H>1$. For $p\ge 2$,
\begin{equ}[e:bounddiffy]
\Bigl\| \sup_{s\le u \le t}\rho(y_u,Y_{s,u}) \Bigr \|_p \lesssim
 \|x\|_{\alpha,p}\cdot  \eps^{-{1\over 2}} |t-s|^{{1\over 2}+\alpha}\;,
\end{equ}
provided that $|t-s| \le \min( \f 14 \delta,  c) \eps$ where $\delta$ is  the injectivity radius of $\CY$ and $ c$ a constant depending on the bounds on $V,V_0$.
\end{lemma}
 \begin{proof}
 Since $\CY$ is compact, we can find a function $d$ which agrees with the Riemannian distance
 in a neighbourhood of the diagonal  and
such that $d^2$ is globally smooth. 
(Take $d=g\circ \rho$ for $g:\R_+\to \R$ a smooth concave function with $g(r) = r$ when $r< \delta/4$ 
and $g(r)= \delta/2$ when $r\ge 3\delta/4$, where $\delta$ denotes the injectivity radius of $\CY$.)

We now claim that, by applying It\^o's formula to $d^{2p}(y_u,Y_{s,u})$
and then using the Burkholder--Davis--Gundy inequality, one obtains for $p \ge 1$ the bound
\begin{equs}
&\E \sup_{s \le u\le t}d^{2p} (y_u,Y_{s,u}) \label{e:boundBDG}\\
&\lesssim \f 1 \eps \int_s^t \E d^{2p}(y_r,Y_{s,r})  dr+
\f 1 \eps \int_s^t \E\left( d^{2p-2}(y_r,Y_{s,r})|x_u-x_r|^2 \right)dr.
\end{equs}
We proceed with the proof based on this, and return to give more explanation at the end of the proof.
Let $\sigma_t=\E \sup_{u \le t}d^{2p}(y_u,Y_{s,u}) $.  Using H\"older's inequality on the last term, 
we obtain the bound 
\begin{equ}
\sigma_t \lesssim \f 1 \eps \int_s^t  \sigma_r^{ \f {2p-2} {2p} }\Bigl(\|x_u-x_r\|_{2p}^{2}
+ \sigma_r^{\f 1p}\Bigr)\,dr \;,
\end{equ}
so that Lemma~\ref{lem:boundFunny} yields
\begin{equ}
\sigma_t^{\f 1p} \lesssim \f 1 \eps \int_s^t \Bigl(\|x_u-x_r\|_{2p}^{2}
+ \sigma_r^{\f 1p}\Bigr)\,dr\;. 
\end{equ}
This allows us to apply Gronwall's inequality, yielding
\begin{equ}
\sigma_t^{\f 1p} \lesssim \f{e^{c\f {t-s} \eps}}{\eps}  \int_s^t \|x_r-x_u\|_{2p}^{2}dr
\lesssim \f{|t-s|}{\eps}  \|x\|_{\alpha,2p}^2|t-s|^{2\alpha }\;,
\end{equ}
where we used the bound $|t-s| \le \eps$
to make sure that the exponential factor doesn't cause an explosion.
This is precisely the required bound since we considered $d^{2p}$ rather than $\rho^{p}$.

The bound \eqref{e:boundBDG} is straightforward in Euclidean space using the $x \leftrightarrow y$ 
symmetry of the distance $|x-y|^2$.  If $\CY$ is a compact manifold, $\<\nabla_x \rho (x, y ), v\>=-\<\nabla_y \rho(x,y), \tilde v\>$, where $v, \tilde v$ are tangent vectors at $T_xM$ and $T_yM$ respectively, and are obtained by parallel translations along the geodesic from one to the other.
This holds because we only consider $x,y$ such that their Riemannian distance $\rho(x,y)$ is smaller
than $1/2$ of the injectivity radius.

This means the stochastic differential $dy_t$ and $dY_{s,t}$ can then be compared using the Lipschitz 
continuity assumption on the vector fields, modulo the Stratonovich correction term which can be 
dealt with by the Lipschitz continuity assumption on $\sum_{i=1}^{\hat m}\nabla _{V_i}V_i$.
The same consideration holds for our modified distance function which is of the form $g\circ \rho$ where $g$ is a smooth real-valued function.  
 \end{proof}

We consider the two parameter family of stochastic processes:
\begin{equs}[e:defAmain]
A_{s,t}^\eps &= \int_s^t h(x_s, Y_{s,r})\,dB_r.
\end{equs}
This is for fixed  process $x_{\fat}\in \CB_{\alpha,p}$
and for $Y_{s,r}=\bar \Phi^{x_s}_{s,t}\big(\Phi_{0,s}^x(y_0)\big)$, as in (\ref{e:defYst}). We will also write $y_t$ for the process
obtained by solving \eqref{fast1}, i.e. $y_t= \Phi^x_{0,t}(y_0,\omega)$,  and we recall that both
$y_t$ and  $Y_{s,t}$  depend of course on $\epsilon$.

We have the following conclusion, which holds for any $x_{\fat}\in \CB_{\alpha,p}$ (we 
emphasise that this does in particular include the solution to (\ref{e:equationx}) but we
do not restrict ourselves to that case).

\begin{lemma}\label{prop:equalA}
Assume that  $h \in \BC^1$ and that $x \in \CB_{\alpha,p}$ with $\alpha + H > 1 + {1\over p}$. 
Then for each $\eps$ fixed,
the process $I_t(A^\eps)$ given by  
Lemma~\ref{sewing-lemma} coincides with the Young integral $\int_0^t h (x_s, y_s)\,dB_s$.
\end{lemma}
\begin{proof}
Note first that under our assumptions, the process $s \mapsto h (x_s, y_s)$ 
belongs almost surely to $\CC^\beta$ for every $\beta < \bigl(\alpha - {1\over p}\bigr) \wedge \f12$,
so that the Young integral is well-defined for every $B \in \CC^{H-\kappa}$ for
$\kappa$ sufficiently small.

Let $\tilde A_{s,t}^\eps= h (x_s, y_s) (B_t-B_s)$, so that $I_t(\tilde A^\eps)$ coincides with 
that Young integral by \eqref{e:defIA}.
As a consequence of the last statement of Lemma~\ref{sewing-lemma},  
it is sufficient to show that 
\begin{equ}
\big\|A_{s,t}^\epsilon-\tilde A_{s,t}^\epsilon\big\|_p
\lesssim |t-s|^{\bar \eta}\;,
\end{equ}
for some $\bar \eta > 1$.
Note that $\eps > 0$ is fixed, so we are allowed to obtain bounds which diverge
as $\eps \to 0$.
Apart from $x_s$, the evolution of $Y_{s,r}$ has no further dependence on the 
fractional Brownian motion, so
that $h(x_s, y_s)-  h(x_s, Y_{s,r})$ is $\CF_s$-measurable for
$\CF_s = \CG_s \vee \sigma(\hat W)$.
We then use Lemma~\ref{lemma-a} with $\kappa=0$  and $p'<p$, so that 
\begin{equ}[e:bounddiffA]
\Big\| \int_{s}^{t}\big(h(x_s, y_s)-  h(x_s, Y_{s,r}) )\big) \ dB_r \Big\|_{p'}
\le \||h(x_s, y_s)-h(x_s, Y_{s,\fat})|_\infty \|_p \,  |s-t|^H\;.
\end{equ}
%
%As in the proof of Lemma~\ref{lem:sewing=Young},
%it suffices to obtain a bound on $\| |f(x_s, y_s^\eps)-f(x_s, Y_{s,\fat})|_\infty \|_p $,
%with the supremum norm running over the interval $[s,t]$,
%of the order $|t-s|^\beta$ for some $\beta$ with $\beta + H > 1$. 
By the Lipschitz continuity of $h$, we have the bound
\begin{equ}
\||h(x_s, y_s)-h(x_s, Y_{s,\fat})|_\infty \|_p 
\lesssim \Big\|\sup_{r \in [s,t]}\rho(y_s, Y_{s,r}) \Big\|_p\;.
\end{equ}
Since the distance $\rho$ is bounded, it follows from  Lemma~\ref{lemma:bounddiffy}
that, for every $\kappa > 0$, one has the bound 
\begin{equ}
\Big\|\sup_{r \in [s,t]}\rho(y_s, Y_{s,r}) \Big\|_p \lesssim 1 \wedge ( \eps^{-{1\over 2}} |t-s|^{{1\over 2}+\alpha})
\le (\eps^{-1/2}|t-s|^{\f 12 +\alpha})^{1 - \kappa}\;.
\end{equ}
Combining this with \eqref{e:bounddiffA}, the claim follows.
\end{proof}
%From the proof of the proposition we extract the following:
%\begin{corollary}
%Let $x_t$ be adapted process with values in $\R^d$. Let $y_t$   and $\{Y_{s,\fat}\}$ be a family of adapted processes on $\CY$  with the following properties :
%$Y_{s,s}=y_s$ and  
%\begin{equ}
%\Big\|\sup_{r \in [s,t]}d(y_s^\eps , Y_{s,r}) \Big\|_p\le  c|t-s|^{\alpha}\;,
%\end{equ}
%where $\alpha>1-H$ and
%$c$ are constants. Suppose that these processes are $\alpha$-H\"older continuous.
%Let $h: \R^d\times \CY\to \R$ be Lipschitz continuous and suppose that  $A_{s,t}\eqdef \int_s^t h(x_s, Y_{s,r} ) \,dB_r$ belongs to $H^p_\eta \cap H^p_{\bar \eta}$ for some $p\ge 2$, $\eta>\f 12$, and $ \bar \eta>1$.
%Then $\lim_{|\CP|\to 0} \sum_{[u,v]\in\CP} A_{u,v} $ with $\CP$ being partitions of $[0,t]$ 
%agrees with the Young integral $\int_0^t h(x_r, y_r) dB_r$.
%\end{corollary}

\subsection{Semi-groups with a parameter: ergodicity and continuity}
\label{sec:semigroup}

Denote by
 $|F|_\Lip$ the best Lipschitz constant of $F$ and set $|F|_\Osc=\sup F-\inf F$. On a  space with bounded radius, $|F|_\Osc \lesssim|F|_\Lip$. 
Recall that $\CP_t^x$ denotes the semigroup associated to \eqref{e:xfrozen}.
 By differentiating the solution flow, a brutal bound on the derivative flow which is the solution to the equation $dv_t=\eps^{-1/2}\nabla _{v_t}V\circ d\hat W_t+\eps^{-1}\nabla_{v_t} V_0\,dt$ (the estimates depend on the covariant derivatives up to order $2$), yields a bound of the type
$|\CP^x_t F|_\Lip \le C e^{Ct/\eps} |F|_\Lip$.  This can be improved using ergodicity, for large time, we will also make use of the smoothing properties of the Markov semigroups~$\CP^x$.

\begin{lemma}\label{lemma-semi-group-0}
Under Assumption~\ref{ass:main}, the following holds.
\begin{enumerate}
\item [(1)]There exist  constants $c, C$ such that for any $x\in \R^d$,
\begin{equs}
|\CP^x_t F|_\Osc &\le C e^{-ct/\eps}  |F|_\Osc\;,\label{osc} \\
\label{e:goodbound}|\CP^x_t F|_\Lip  &\le C e^{-ct/\eps}  |F|_\Lip\;,\\
|\CP^x_t F|_{\BC^2} &\le C \eps t^{-1} e^{-ct/\eps}  |F|_{\infty}\;,
\label{e:boundBC2}
\end{equs}
uniformly over $t\in \R_+$. 
\item [(2)]
For any $x, \bar x \in \R^d$, the bound
\begin{equ}\label{e:telescopic}
|\CP_t^x F - \CP_t^{\bar x}F|_\infty
\le  C |x-\bar x| \,|F|_\Lip\;,
\end{equ}
holds for all $t \ge 0$.
\item[(3)] If $h: \R^d\times \CY\to \R$ is a Lipschitz continuous bounded function with  $\int_\CY h(x,y)\mu^x(dy)=0$ for every $x$, then for any $\kappa\in (0,1)$,
\begin{equ}[e:goodBoundSG]
|\CP_{t}^{x} h(x,\fat) - \CP_{t}^{\bar x} h(\bar x,\fat)|_\infty
\lesssim  |h|_\infty^{\kappa}|h|_\Lip^{1-\kappa} |x-\bar x|^{1-\kappa} e^{-\kappa ct/\eps} \;.
\end{equ}
\end{enumerate}\end{lemma}

\begin{proof}
It follows from standard estimates, see for example \cite{BEL-Elworthy-Li}, that  
for any $t\in (0,\eps)$, one has the bounds
\begin{equ}[e:smoothing]
|\CP^x_t F|_\Lip \le C(1 \vee (t/\eps)^{-1/2})  |F|_\Osc\;,\qquad 
|\CP^x_t F|_{\BC^2} \le C(1 \vee (t/\eps)^{-1})  |F|_\infty\;,
\end{equ}
where the constant $C$ only depends on derivatives of $V_0$ up to order $2$ and of the 
remaining $V_k$ up to order $3$. 
The reason why one obtains the oscillation norm on the right hand side of the first 
bound is that $|F|_\Lip$ does not change under constant shifts and 
$|F|_\Osc=\inf_{c \in \R} 2 |F-c|_\infty$. In fact (\ref{e:smoothing}) holds for all $t$, but we do not need it.
%Since the modulus of the derivative flow and its derivative grow exponentially with time, this is % shown initially for any $t\in (0,\eps)$.
%We can then use the contraction property of the semigroup to extend it to the whole
%time interval, for the first inequality this is
%$$\begin{aligned}|P_t^x F(y)-P_t^x F(y')|
%&=|P_{t-[\f t \epsilon] \epsilon}^x (P_{[\f t \epsilon] \epsilon}^x F)(y)-P_{t-[\f t \epsilon]
%\epsilon}^x (P_{[\f t \epsilon] \epsilon}^x F)(y')|
%\\&\le  C(1 \vee (t/\eps)^{-1/2}) |P_{[\f t \epsilon] \epsilon}^x F|_{\infty}. 
%\end{aligned}$$

It furthermore follows from the uniform positive lower bounds on the heat kernel  
(see \cite{Aronson} for the case of $\R^n$ and for example \cite{Cheeger-Yau-81,Qian} for versions 
that apply to manifolds, see also \cite{FengYu})
that $|\CP_\eps^x F(y_1) -\CP_\eps^x F(y_2)|\le (1-\lambda) |F|_\Osc$ for some constant $\lambda>0$, 
so that \eqref{osc} follows by iterating this bound (Doeblin's condition).

Using this  last inequality for time $t-\epsilon$ and  (\ref{e:smoothing})
for the remaining time $\eps$, 
we obtain for $t \ge \eps$ the bound 
$$|\CP_t^x F|_\Lip=|\CP_{\epsilon }^x \CP_{(t-\epsilon)}^xF|_\Lip \le C e^{-ct/\eps}  |F|_\Lip.$$
For $t \le \eps$ on the other hand, this bound follows from $L^p$ bounds on the
Jacobian, so that \eqref{e:goodbound} holds. The bound \eqref{e:boundBC2} follows in the same way.

It is also rather straightforward to verify that 
\begin{equ}[e:badbound]
|\CP_t^x F - \CP_t^{\bar x}F|_\infty \le Ce^{Ct/\eps} |x-\bar x| \,|F|_\Lip \;.
\end{equ}
While this bound is good for $t \le \eps$, it can be significantly improved for $t > \eps$. 
Indeed, for $t \ge \eps$ and any partition $\Delta$ of $[0,t]$ into subintervals of size at most $\eps$ and at least $\eps/2$, we have
\begin{equs}
|\CP_t^x F - \CP_t^{\bar x}F|_\infty
&\le \sum_{[s,u] \in \Delta}
|\CP_{t-u}^x(\CP_{u-s}^x - \CP_{u-s}^{\bar x})\CP^{\bar x}_{s} F|_\infty \\
&\lesssim \sum_{[s,u] \in \Delta} e^{-cs/\eps} |x-\bar x| \,|F|_\Lip
\le  
C |x-\bar x| \,|F|_\Lip\;.
\end{equs}
Here we used \eref{e:goodbound}, the small time bound \eref{e:badbound},
and the fact that $\CP_{t-u}^x$ is a contraction in $L^\infty$.

For the last bound we make use of the fact that, since the integral of $h(x)$
against the invariant measure for $\CP_{t}^{x}$ vanishes, its supremum norm is controlled by its oscillation and vice versa, so that
(\ref{osc}) yields the bound
\begin{equ}
|\CP_{t}^{x} h(x,\fat)|_{\infty} \le C e^{-ct/\eps} |h|_\infty\;.
\end{equ}
Combining this with (\ref{e:telescopic}), we see that
%the last estimate of Lemma~\ref{lemma-semi-group-0} yields
\begin{equs}
|\CP_{t}^{x} h(x,\fat) - \CP_{t}^{\bar x} h(\bar x,\fat)|_\infty
&\lesssim  \inf\{|h|_\infty e^{-ct/\eps} , |x-\bar x|\,|h|_\Lip\} \\
&\lesssim |h|_\infty^{\kappa}|h|_\Lip^{1-\kappa} |x-\bar x|^{1-\kappa} e^{-\kappa ct/\eps} \;,
\end{equs}
completing the proof.
\end{proof}

\begin{lemma}\label{lem:boundNegHolderGen}
Fix $\bar x \in \R^d$ and $y \in \CY$ and write $y_t = \bar \Phi_{0,t}^{\bar x}(y)$.
Fix $p \ge 2$ and $F \colon \CY \to \R$ bounded measurable, and write
$\bar F (\bar x)= \int F(y)\,\mu^{\bar x}(dy)$.
Then, one has the bound
\begin{equ}[e:boundErgodic]
\Big\|\int_0^t \big(F(y_r) - \bar F(x)\big)\,dr\Big\|_p
\lesssim \eps^{1 \over p}\, t^{1-{1\over p}}|F|_\Osc\;,
\end{equ}
uniformly over $\bar x \in \R^d$, $y \in \CY$, and $t \ge 0$.
\end{lemma}

\begin{proof}
Since $\inf F \le \bar F(x)\le \sup F$, one has the almost sure bound 
\begin{equ}
\Big|\int_0^t \big(F(y_r) - \bar F(x)\big)\,dr\Big| \le 2 t|F|_\Osc \;,
\end{equ}
so that the general case of \eqref{e:boundErgodic} follows from the case $p=2$ by interpolation. 

For $p=2$, we write $\tilde F(y) = F(y) - \bar F(\bar x)$. With this notation, we have the identity
\begin{equ}
\E\Big(\int_0^t \tilde F(y_s)\,ds\Big)^2 =
2\int_0^t\int_0^r \bigl(\CP_{s}^{\bar x}(\tilde F \cdot \CP_{r-s}^{\bar x} \tilde F) \bigr)(y)\,ds\,dr\;.
\end{equ}
By Lemma~\ref{lemma-semi-group-0}, we can bound the supremum of 
$\CP_{r-s}^{\bar x} \tilde F$ by $Ce^{-c|r-s|/\eps}|\tilde F|_\Osc$,
giving the bound 
\begin{equ}
\E\Big(\int_0^t \tilde F(y_s)\,ds\Big)^2 
\le C|\tilde F|_\Osc^2 \int_0^t\int_0^r e^{-\f{c|s-r|}\eps} \,ds\,dr\;.
\end{equ}
Since this integral is bounded by a multiple of $\eps t$
and since $|\tilde F|_\Osc = |F|_\Osc$, the claim follows.
\end{proof}

\begin{corollary}\label{lem:boundNegHolder}
We fix an adapted stochastic process $x_t$ with values in $\R^d$ and
 $y_t$ with values in $\CY$. Let $h: \R^d\times \CY\to \R$ be bounded measurable, 
 set $Y_{s,t}=\bar \Phi_{s,t}^{x_s}(y_s)$, and set $\bar h(x) = \int h(x,y)\mu^x(dy)$.
For $p \ge 2$ and $s\le u \le t$, one has the bound 
\begin{equ}
\Big\|\int_u^t \big(h(x_s, Y_{s,r}) - \bar h(x_s)\big)\,dr\Big\|_p
\lesssim |h|_\Osc \, \eps^{1 \over p}\, |t-u|^{1-{1\over p}}\;,
\end{equ}
uniformly over $\eps \in (0,1]$, and over the processes $x_t$ and $y_t$.
\end{corollary}

\begin{proof}
Applying Lemma~\ref{lem:boundNegHolderGen} with $\bar x = x_s$, $y = Y_{s,u}\equiv \bar \Phi_{s,u}^{x_s}(y_s)$ and
$F = h(x_s,\fat)$, we then obtain  
\begin{equ}
\E\left( \left| \int_u^t \big(h(x_s, \bar \Phi_{u,r}^{x_s} (Y_{s,u})) - \bar h(x_s)\big)\,dr\right|^p \, \Big |\,\CF_u\right) \lesssim |h(x_s, \fat)|_\Osc \, \eps^{1 \over p}\, |t-u|^{1-{1\over p}}\;.
\end{equ}
Since this bound is uniform over
$y$ and since $|h(\bar x, \fat)|_\Osc\le |h|_\Osc$ for every $\bar x$, the claim follows.
\end{proof}

\subsection{Regularity of limiting drift}
\label{sec:regIM}

In this section, we show that Assumption~\ref{ass:main} guarantees that the
limiting functions $\bar f$ and $\bar g$ do again belong to $\BC^2$.
Throughout Section~\ref{sec:regIM} (and only here), we write $\CP_t^x$ for the semigroup generated
by \eqref{e:xfrozen}, but with $\eps = 1$, i.e.\ we work on the fast timescale.
The reason why we can do this is that we are only interested in showing a result about the
invariant measures, and these do not depend on $\eps$.
We also write $\bar \Phi^x_t$ for the corresponding flow map, so that 
$\bigl(\CP_t^x F\bigr)(y) = \E F(\bar \Phi^x_t(y))$.
The main ingredient for the proof is the following claim.

\begin{lemma}\label{lem:smoothPtx}
Under Assumption~\ref{ass:main}, for any fixed $\tau > 0$,
the map $x \mapsto \CP_{ \tau}^x$ is differentiable, uniformly in $x$,
as a map $\R^d \to L(\BC^2,\BC^1)$ and as a map $\R^d \to L(\BC^1,\BC^0)$.
It is also twice differentiable, uniformly in $x$,
as a map $\R^d \to L(\BC^2,\BC^0)$.
\end{lemma}

\begin{proof}
Consider the semigroup $\bar \CP_t$ on $\R^d \times \CY$ given by
\begin{equ}
\bigl(\bar \CP_t F\bigr)(x,y) = \E F(x,\bar \Phi^x_t(y))\;.
\end{equ}
By \cite[Prop.~A8]{BEL-Elworthy-Li} and \cite{Li-p-complete} (see also \cite{Kunita-book}
for related results), 
$\bar \CP_t$ maps $\BC^k$ into itself for
$k \in \{0,1,2\}$, and the claim follows.
\end{proof}

\begin{lemma}\label{lem:regularfbar}
Under Assumption~\ref{ass:main}, the map $\bar f(x) = \int_\CY f(x,y)\,\mu^x(dy)$ belongs to~$\BC^2$.
\end{lemma}

\begin{proof}
For any given $x$ and $t$, we view $\CP_t^x$ as a bounded linear operator on 
the spaces $\BC^k$ of $k$ times continuously differentiable functions on $\CY$,
$k=0,1,2$. 
Writing $\Pi_x$ for the projection operator given by 
$\Pi_x F = \scal{\mu^x,F}\one$, where $F: \CY\to \R$ and $\one$ denotes the constant function, we
note that $\Pi_x$ commutes with $\CP_t^x$ and that,
by part (1) of Lemma~\ref{lemma-semi-group-0}, we can choose $t$ sufficiently 
large so that
$|(1-\Pi_x)\CP_t^x F|_{\BC^k} \le {1\over 2}|F|_{\BC^k}$ for $k \in \{0,1,2\}$, 
uniformly over $x$. We used the fact that  $F-\Pi_x F$ is centred.

We fix such a value of $t$ from now on.
Writing $R^x(\lambda)$ for the resolvent of $\CP_t^x$, it follows that
that for each $k \in \{0,1,2\}$,
the operator norm of $R^x(\lambda)$ in $\BC^k$ is bounded by $4$, 
uniformly in $x$ and uniformly over $\lambda$
belonging to the circle $\gamma$ of radius ${1\over 4}$ centred at $1$.
Indeed, for $B = \BC^k$, we write $B=\scal{\one}\oplus B^\perp$
where $B^\perp =\{F: \<\mu^x, F\>=0\}$, and view $\Pi_x$ as the projection onto
$\scal{\one}$ for this decomposition.

Since $\Pi_x$ commutes with $\CP_t^x$, that operator splits with respect to this decomposition
as $\CP_t^x=\id \oplus (1-\Pi_x)\CP_t^x$ and its resolvent is given by
$$R^x(\lambda) = (\lambda-\CP^x_t)^{-1}=(\lambda-\id)^{-1} \oplus (\lambda-(1-\Pi_x)\CP_t^x)^{-1}\;.$$
The first term is obviously bounded by $4$, while the second term is given by
the convergent Neumann series
$$\lambda^{-1} \Big(1+\f 1 \lambda (1-\Pi_x)\CP_t^x+\f 1 {\lambda^2} ((1-\Pi_x)\CP_t^x)^2+\dots\Big)\;,$$
which is also bounded by $4$ in operator norm since $|\lambda| \ge \f34$ and $\|(1-\Pi_x)\CP_t^x\|\le \f12$.
We claim that, uniformly over $\lambda \in \gamma$ and $x \in \R^d$,
the map $x \mapsto R^x(\lambda)$ is $\CC^2$ as a map into $L(\BC^2,\BC^0)$
and $\CC^1$ as a map into $L(\BC^2,\BC^1)$ and into $L(\BC^1,\BC^0)$.
Indeed, this is an immediate consequence of the identities
\begin{equs}[e:boundsDR]
D_x R^x &= R^x(D_x \CP_t^x) R^x \\
D_x^2 R^x &= 2 R^x (D_x \CP_t^x) R^x (D_x \CP_t^x) R^x
+ R^x (D_x^2 \CP_t^x) R^x\;,
\end{equs}
combined with Lemma~\ref{lem:smoothPtx}.

We now recall that one has \cite[Thm~III.6.17]{Kato}
\begin{equ}
\Pi_x = {1\over 2i\pi} \oint_\gamma R^x(\lambda) \,d\lambda\;.
\end{equ}
In particular, for any fixed probability measure $\mu$ on $\CY$, one has the identity
\begin{equ}[e:resolventfbar]
\bar f(x) = {1\over 2i\pi} \oint_\gamma \scal{\mu,R^x(\lambda)f(x,\fat)} \,d\lambda\;.
\end{equ}
It now suffices to note that, by our assumptions, one has
$D_x^k f(x,\fat) \in \BC^{2-k}$ for $k \in \{0,1,2\}$ uniformly over $x$, and therefore
the claim follows from \eqref{e:boundsDR} combined with Lemma~\ref{lem:smoothPtx}.
\end{proof}

\subsection{Uniform estimates on the slow variable}
\label{sec:uniformBounds}

The main theorem in this section is a uniform estimate for the slow variables. 
The divergence of the H\"older norm of $r \mapsto y_r^\eps$ means
that the bounds given on $A^\eps$ and $\tilde A^\eps$ in the proof of 
Lemma~\ref{prop:equalA} diverge as $\eps \to 0$, 
and are thus inadequate to show any kind of tightness result.
This is where our precise choice of $A^\eps$ comes in.
We will show that $A^\eps$ belongs to the Banach space $H^p_\eta\cap \bar H^p_{\bar \eta} $
with uniform (in $\eps$) upper bounds on its norms.

The estimates obtained in Lemma~\ref{lemma-semi-group-0} on $\CP_t^x$ are not quite
sufficient for our use, 
we will also introduce a second family of 
random semigroups $\CQ_{s,t}^x$ generated by the flow $\Phi^x_{s,t}$, 
see Definition~\ref{def:Qst} below.
Given the Wiener process $\hat W_t$ and the fractional Brownian motion $B_t$ as before, 
we define the filtration $\CF_t = \CG_t \vee \hat \CG_t$, where
\begin{equ}
\CG_t = \sigma\{ B_u-B_r\,:\, r\le  u\le t \}\;,\qquad
\hat \CG_t = \sigma\{\hat W_u-\hat W_r\,:\, r\le u\le t\}\;.
\end{equ}
Observe that 
$\CG_t= \sigma\{ W_u-W_r\,:\, r\le  u\le t \}$.
We will also make use of the `noise'
\begin{equ}[e:defNoise]
\hat \CG_t^s = \sigma \{\hat W_u-\hat W_r\,:\, s \le r\le u\le t\}\;,
\end{equ}
and also $ \CG_t^s$ defined using $W_t$.
In the definition for  $\CB_{\alpha,p}$, we use the filtration $\CF_t$ given above.

The following theorem is the main technical tool in the proof of Theorem~\ref{thm:main} as it
yields uniform bounds in $\eps$ on the fixed point map defining $x$. 
Its proof relies on three further lemmas, Lemmas~\ref{lemma:Q},~\ref{lemma:semi-group-2}, and~\ref{lemma:QP},
 given after the  proof of the theorem. 
% This theorem concerns only the fast variable equation and an auxiliary process $x_{\fat}$,
% hence no assumptions are needed for $f$, $g$. 
\begin{theorem}
\label{fixed-point-map}
Let $T>0$, let $p \ge 2$ and $\alpha \in (0,\f12)$ be such that  $1+ \f 1 p<\alpha+H$,
 let $x_{\fat}$ be an $\R^d$-valued process in $\CB_{\alpha,p}$, 
 and let $y_{\fat}^\eps$ be the solution to (\ref{fast1}) where
$V$ is assumed to satisfy Assumption~\ref{ass:main}.
Then there are exponents $\eta > {1\over 2}$ and $\bar \eta > 1$ such that, 
for $h: \R^d\times \CY\to \R$  a bounded uniformly Lipschitz continuous function, one
has the following.
\begin{enumerate}
\item Suppose in addition that $\bar h(x) \eqdef \int_\CY h(x, y)\mu^x(dy)=0$
for all $x$.  Then  for any $\beta<H$, there exists a constant $\kappa>0$ 
such that
\begin{equ}[e:wantedUniformBound]
\left\|\int_s^t h(x_r, y_r^\eps) \, dB_r\right\|_{p} \lesssim
 \eps^\kappa   |h|_{\BC^1}  \left( \|x\|_{\alpha, p} |t-s|^{\bar \eta} 
+ |t-s|^{\eta}\right)\;,
\end{equ}
holds uniformly over $\eps$ and over $x_t$.
\item For general $h$, one has the bound
\begin{equ}[e:wantedUniformBoundSimple]
\left\| \int_s^t h(x_r, y_r^\eps) \, dB_r\right\|_{p}  \lesssim   |h|_{\BC^1} \left( \|x\|_{\alpha, p} |t-s|^{\bar \eta} 
+ |t-s|^{\eta}\right)\;.
\end{equ}
\end{enumerate}\end{theorem}

\begin{proof}
Fix an arbitrary $\R^d$-valued  stochastic process $x_t$ in $\CB_{\alpha,p}$ (which is not necessarily a solution
to our equation).
We first note that the bound \eqref{e:wantedUniformBoundSimple} for general $h$ (i.e.\ with $\bar h\not =0$) 
follows from the bound \eqref{e:wantedUniformBound}, combined with the estimate  (\ref{e:boundCA}) 
for $\int_0^t \bar h(x_r) dB_r$. We therefore only focus on the proof of \eqref{e:wantedUniformBound}
and assume that $\bar h=0$ from now on.

During the rest of the section, we will also suppress the 
superscript $\eps$ whenever possible. 
Note that standard estimates for the Young integral $\int_0^t h(x_r, y_r^\eps) \, dB_r$ 
 obtained by taking limits in the Riemann sum $ \sum _{[v,u]\in \CP} h(x_s, y_s)(B_u-B_v)$
 would require uniform (in $\eps$) bounds on the H\"older norms on $y^\epsilon$, which
 we do not have.
 
In order to obtain estimates that are uniform in $\eps$, it will be useful to use Lemma~\ref{prop:equalA} and write the integral as
$\lim_{|\CP| \to 0} \sum _{[u,v]\in \CP} A_{u,v}$,
where $A$ is given~by $$A_{s,t}=\int_s^t h(x_s, Y_{s,r}) \, dB_r.$$

In order to obtain the bound \eqref{e:wantedUniformBound} from Lemma~\ref{sewing-lemma}, it therefore remains
to obtain a bound on $\|A\|_{\eta,p}$ and $\$A\$_{\bar \eta,p}$
for some $p \ge 2$, some $\eta > {1\over 2}$ and some $\bar \eta > 1$.
The bound on $\|A\|_{\eta,p}$ is contained in the following lemma, the proof of which is relatively straightforward.

\begin{lemma}\label{lem:boundAst}
Assume that $\bar h=0$. For every $p\ge 2$ and $\kappa > 0$, one has $\{A_{s,t}\}\in H^p_\eta$ with $\eta = H-\kappa$ and
\begin{equ}
\|A_{s,t}\|_{p} \lesssim \eps^{\bar \kappa} |t-s|^{H-\kappa}\;,
\end{equ}
uniformly over $s,t \in [0,T]$, provided that $\bar \kappa < (\kappa/2)\wedge (1/p)$.
\end{lemma}

\begin{proof}
To prove that $\|A\|_{\eta, p}=\sup_{s<t}\f{\|A_{s,t}\|_p}{|t-s|^\eta}$ is finite,  we 
first write
\begin{equ}
F_s(r) = h(x_s, Y_{s,r})\;,
\end{equ}
so that 
\begin{equ}
A_{s,t} = \int_s^t F_s(r)\,dB_r\;.
\end{equ}
Furthermore, conditional on $\CF_s$, $F_s$ and $B$ are independent,
so that we can apply Lemma~\ref{lemma-a}, yielding for $q > p$  and $\kappa\in [0,1)$ the bound
\begin{equ}
\|A_{s,t}\|_p \lesssim \| |F_s|_{-\kappa} \|_{q} \,|t-s|^{H-\kappa}\;. 
\end{equ}
On the other hand, it follows from Corollary~\ref{lem:boundNegHolder} exploiting ergodicity,
that for $s \le u \le v\le t$, one has the bound
\begin{equ}
\Big\|\int_u^v F_s(r)\,dr\Big\|_q \lesssim \eps^{1\over q} |v-u|^{1-{1\over q}}\;,
\end{equ}
so that Kolmogorov's continuity theorem implies the bound
\begin{equ}
\||F_s|_{-\kappa}\|_{q}= \Big\| \sup_{u \not = v} |u-v|^{1-\kappa} \int_u^v F_s(r)\,dr\Big\|_q \lesssim \eps^{1\over q}\;,
\end{equ}
provided that $\kappa > {2\over q}$. Choosing $q = 1/ \bar\kappa$ completes the proof. 
\end{proof}

It now remains to show that $A_{s,t}\in \bar H^p_{\bar \eta}$ for some $\bar \eta>1$ and to obtain a suitable
bound for small values of $\eps$.
We have the identity
\begin{equ}
\delta A_{sut} = \int_u^t \big(h(x_s, Y_{s,r})-h(x_u, Y_{u,r})\big)\, dB_r.
\end{equ}

Recall that this integral is defined as the sum of a Wiener integral against $\tilde B_r^u$ and a 
Riemann--Stieltjes integral against $\bar B_r^u$. Since $\tilde B_r^u$  is independent of 
$\CF_u\vee \hat \CG_t$,  while $h(x_s, Y_{s,r})-h(x_u, Y_{u,r})$ is measurable with respect to it,
the Wiener integral has vanishing conditional
expectation against $\CF_u$, so that 
\begin{equs}
\E (\delta A_{sut} \, \big|\, \CF_u) &= \int_u^t \E \big(h(x_s, Y_{s,r})-h(x_u, Y_{u,r})\,\big|\, \CF_u\big)\,\dot {\bar B}^u_r\,dr\;\\
&= \int_u^t \big(  \CP_{r-u} ^{x_s}  h(x_s, \fat) (Y_{s,u} )-\CP_{r-u}^{x_u}  h(x_u, \fat)(y_u)\big)\,\dot {\bar B}^u_r\,dr.
\end{equs}
In the last line we have used
$$Y_{s,r}=\bar \Phi_{s,r}^{x_s} (y_s)=\bar \Phi_{u,r}^{x_s}\bar \Phi_{s,u}^{x_s}(y_s)
=\bar \Phi_{u,r}^{x_s} (Y_{s,u})\;,\qquad
Y_{u,r}=\bar \Phi_{u,r}^{x_u} (y_u)\;.
$$
It seems to be difficult to get a good enough bound on this expression, so we exploit the fact that 
we really only need to bound the conditional expectation of $\delta A_{sut}$ with 
respect to $\CF_s$ rather than $\CF_u$. Conditioning on $\CF_s$ however has the 
unfortunate side effect that it no longer keeps this term separate from the term
$\dot{\bar B}_r^u$. Instead, we are going to condition on
$\CF_s \vee \CG_u$. We have 
\begin{equs}
 &\E (  \delta A_{sut}\,\big| \, \CF_s \vee\CG_u) \label{e:condDeltaA}\\
&=\int_u^t  \E\left(  \CP_{r-u} ^{x_s}  h(x_s, \fat) (Y_{s,u} )-\CP_{r-u}^{x_u}  h(x_u, \fat)(y_u) \,\big| \, \CF_s \vee \CG_u\right)
 \,\dot {\bar B}^u_r\,dr\\
 &=\int_u^t  \left( \CP^{x_s}_{y-s} \CP_{r-u} ^{x_s}  h(x_s, \fat) (y_s) -\E\left( \CP_{r-u}^{x_u}  h(x_u, \fat)(y_u) \,\big| \, \CF_s \vee \CG_u\right)\right)
 \,\dot {\bar B}^u_r\,dr. \end{equs}
 Given  a `final time' $u$, we write
\begin{equ}\label{CB}
\CU_s^u = \{F \colon \Omega \times \CY \to \R\,:\, \text{$F$ is bounded and $(\CF_s \vee \CG_u)\otimes \CB(\CY)$-measurable}\}\;.
\end{equ}
Given an element $F \in \CU_s^u$, we will use various norms for its $\CY$-dependency, but will always keep the $\omega$-dependency
fixed, so that these norms are interpreted as $\R_+$-valued random variables. For example, we set
\begin{equ}
|F|_\infty(\omega) = \sup_{y\in \CY}|F(\omega,y)|\;,\quad
|F|_\Lip(\omega) = \sup_{\bar y\neq y\in \CY} \rho(y,\bar y)^{-1} |F(\omega,y)-F(\omega,\bar y)|\;,
\end{equ}
and similarly for $|F|_\Osc(\omega)$, but we will always denote them simply by $|F|_\infty$, $|F|_\Lip$, etc.

For any stochastic process $x$ (not necessarily a solution
to our equation) adapted to the full filtration $\CF$, we then define a collection of bounded linear operators
$\CQ_{r,v}^x \colon \CU_v^u \to \CU_r^u$ in the following way. %with $r \le v \le u$ 
\begin{definition}\label{def:Qst}
Given a fixed value $u$ and a process $x$ adapted to $\CF$, we set for $r \le v \le u$ and $F\in \CU_v^u$,
\begin{equ}\label{Q-semi-group}
\bigl(\CQ_{r,v}^x F\bigr)(\omega,y) \eqdef \E \bigl(F(\fat,\Phi_{r,v}^{x}(y,\fat))\,|\, \CF_r \vee \CG_u\bigr)(\omega)\;.
\end{equ}
\end{definition}

\begin{remark}
The fact that we have $\CG_u$ and not $\CG_v$ in the right hand side of \eqref{Q-semi-group} is not a typo!
We always condition on the whole trajectory of the fractional 
Brownian motion $B$ up to the `final' time $u$. Observe that  $\CQ_{r,v}^x F$ is a three parameter family of stochastic processes, and could be denoted by $\CQ_{r,v}^{x, u} F$.
\end{remark}

Since $\bar \Phi^{\bar x}_{r,v}(y)$ is independent of $\CG_u$, for $u\ge v\ge r$, we can also build from $\CP_{v-r}^{\bar x}$ 
an operator $\hat \CP_{r,v}^{\bar x} \colon \CU_v^u\to \CU_r^u$ by setting
\begin{equ}
\bigl(\hat \CP_{r,v}^{\bar x} F\bigr)(\omega,y) 
\eqdef \E \bigl(\bigl( \CP_{v-r}^{\bar x} F\bigr)(\fat,y) \,|\,\CF_r\vee \CG_u\bigr)(\omega)\;.
\end{equ}
By $\CP_{v-r}^{\bar x} F( \omega,y) $ we mean  applying the semigroup to each $F(\omega,\cdot)$, so if $F$ happens to be $\CF_r \vee \CG_u$-measurable, then 
$\hat \CP_{r,v}^{\bar x}$ coincides with $\CP_{v-r}^{\bar x}$,
applied $\omega$-wise.

Using these notations, we can rewrite \eqref{e:condDeltaA} as
\begin{equ}[e:exprCondA]
 \E (  \delta A_{sut}\,\big| \, \CF_s \vee\CG_u)
=\int_u^t  \left( \hat \CP_{s,u} ^{x_s} \CP_{r-u} ^{x_s}  h(x_s, \fat) (y_s)-\CQ^x_{s,u} \CP_{r-u}^{x_u}  h(x_u, \fat)(y_s)\right)
 \,\dot {\bar B}^u_r\,dr. 
\end{equ}
%In order to bound this expression, we will repeatedly make use of the fact that, since the integral of $h(x)$
%agains the invariant measure for $\CP_{t}^{x}$ vanishes, and its supremum norm is controlled by its oscillation,
%(\ref{osc}) yields the bound
%\begin{equ}
%|\CP_{t}^{x} h(x,\fat)|_\infty \le C e^{-ct/\eps} |h|_\infty\;.
%\end{equ}
%Combining this with the last estimate of Lemma~\ref{lemma-semi-group-0} yields
%\begin{equs}[e:goodBoundSG]
%|\CP_{t}^{x} h(x,\fat) - \CP_{t}^{\bar x} h(\bar x,\fat)|_\infty
%&\lesssim  |h|_\infty e^{-ct/\eps} \wedge |x-\bar x|\,|h|_\Lip \\
%&\lesssim ( |h|_\infty)^{\kappa}|h|_\Lip^{1-\kappa} |x-\bar x|^{1-\kappa} e^{-\kappa ct/\eps} \;.
%\end{equs}

The expression in \eqref{e:exprCondA} then naturally splits into two parts. 
The first part is given by 
 $$I_1 \eqdef \int_u^t  \hat \CP_{s,u} ^{x_s} \big(   \CP_{r-u} ^{x_s}  h(x_s, \fat) -  \CP_{r-u} ^{x_u}  h(x_u, \fat) \big)(y_s)\dot {\bar B}^u_r\,dr\;,$$
We then apply the estimate in   \eqref{e:goodBoundSG}, namely
\begin{equs}
|\CP_{t}^{x} h(x,\fat) - \CP_{t}^{\bar x} h(\bar x,\fat)|_\infty
\lesssim  |h|_\infty^{\kappa}|h|_\Lip^{1-\kappa} |x-\bar x|^{1-\kappa} e^{-\kappa ct/\eps} \;.
\end{equs}
This term is then bounded by 
\begin{equ}\label{estimate-A1}
|I_1|
\lesssim |h|_\Lip^{1-\kappa}\, |h|_\infty^\kappa\, \E \big( |x_s- x_u|^{1-\kappa} \,|\, \CF_s\vee \CG_u\big) \int_u^t e^{-\kappa c(r-u)/\eps}  |\dot {\bar B}^u_r|\,dr\;.
\end{equ}
Consequently,  for $\f 1 {p'}+\f 1{q'}=1$, \begin{equ}
\|I_1\|_{p}
\lesssim |h|_\Lip^{1-\kappa}\, |h|_\infty^\kappa\,
\bigl\| \E \big( |x_s- x_u|^{1-\kappa} \,|\, \CF_s\vee \CG_u\big)\bigr\|_{pp'} 
\Bigl\|\int_u^t e^{-{\kappa c(r-u)\over \eps}}  |\dot {\bar B}^u_r|\,dr\Bigr\|_{pq'}\;.
\end{equ}
We choose $p' = (1-\kappa)^{-1}$ and $q'=\kappa^{-1}$, which yields the bound
\begin{equ}
\bigl\| \E \big( |x_s- x_u|^{1-\kappa} \,|\, \CF_s\vee \CG_u\big)\bigr\|_{pp'} 
\le \|x\|_{\alpha,p}  |s-u|^{(1-\kappa)\alpha} \;.
\end{equ}
recall that $\|\dot {\bar B}^u_r\|_q \lesssim |r-u|^{H-1}$ for every $q \ge 1$ so that 
$$\Bigl\|\int_u^t e^{-{\kappa c(r-u)\over \eps}}  |\dot {\bar B}^u_r|\,dr\Bigr\|_{pq'}
\lesssim \eps^H \wedge |t-u|^H.$$
Combining these bounds, we conclude that for every $\bar \eta< H+\alpha$ there exist $\kappa, \bar \kappa > 0$ such that  
\begin{equ}
\|I_1\|_{p}\lesssim  \epsilon^{\bar \kappa}  \|x\|_{\alpha,p} \, |t-u|^{\bar \eta}\,   |h|_\Lip^{1-\kappa}\, |h|_\infty^\kappa\;.
\end{equ}

The remaining term is given by 
\begin{equ}
I_2 \eqdef \int_u^t  \big(\bigl(\hat \CP_{s,u} ^{x_s} - \CQ_{s,u} ^{x}\bigr) \CP_{r-u} ^{x_u}  h(x_u, \fat)\big) (y_s)\big)\dot {\bar B}^u_r\,dr\;.
\end{equ}
We then apply  the following estimate from Lemma~\ref{lemma:QP}, to be found after this proof, 
%\label{fromQP}
%\CQ_{s,r}^x \circ \CQ_{r,v}^x&=\CQ_{s,v}^x\\
$$|\hat \CP_{s,u}^{x_s}F - \CQ_{s,u}^x F|_\infty 
\lesssim  \sqrt{\E (|x|_{\bar \alpha}^2 \,|\,\CF_s \vee \CG_u)}\, |u-s|^{\bar \alpha} |F|_\Lip\;,$$
where the H\"older norm $|x|_{\bar \alpha}$ is taken on $[s, u]$ and $\bar \alpha<H$ is a number to be chosen, to deduce that
\begin{equs}
\big|I_2\big|
 &=\left| \int_u^t  \big( 
\hat \CP^{x_s}_{s,u}  \CP_{r-u} ^{x_s}  h(x_s, \fat) -  \CQ^{x}_{s,u} \CP_{r-u}^{x_s}  h(x_s, \fat)   \big)(y_s )
 \, \dot {\bar B}^u_r\,dr \right| \\
 &\lesssim \int_u^t  \sqrt{\E (|x|_{\bar \alpha}^2 \,|\,\CF_s \vee \CG_u)}\, |u-s|^{\bar \alpha}  |\CP_{r-u} ^{x_s}  h(x_s, \fat) |_\Lip  
 \, |\dot {\bar B}^u_r|\,dr\;.   
 \end{equs}
We apply to this the following estimate obtained in Lemma~\ref{lemma-semi-group-0}:
$$ |\CP_{r-u} ^{x_s}  h(x_s, \fat)  |_\Lip  \le C e^{-c(r-u)/\eps}  |h(x_s, \fat) |_\Lip
\le C e^{-c(r-u)/\eps}  |h|_\Lip\;.$$
Then, provided that we choose $\bar \alpha$ and $p$ in such a way that  $\alpha < \bar \alpha- {1\over p}$,
we can apply Kolmogorov's continuity theorem yielding
$$\|I_2\|_p \lesssim \|x\|_{\alpha, p}\epsilon^{\bar \kappa}  |h|_\Lip
\, (t-s)^{\bar \alpha +H -\bar \kappa}.$$
Combining these estimates, we have shown that, provided that we choose $p$ sufficiently large and $\bar \kappa$
sufficiently small, there exists $\bar\eta>1$ and a constant $C(h)$ such that 
\begin{equ}
\left\| \E (  \delta A_{sut}\,\big| \, \CF_s \vee \CG_u)\right\|_{p} \le
C(h) \|x\|_{\alpha, p}\epsilon^{\bar \kappa}  \, (t-s)^{\bar \eta}\;.
\end{equ}

When combining this with Lemma~\ref{lem:boundAst}, we have proved that $A$ belongs to 
 $H^p_\eta\cap \bar H^p_{\bar \eta}$ with $\eta>\f 12$ and $\bar \eta>1$, and we have
 obtained bounds for it that are of order $\eps^{\bar \kappa}$ for sufficiently small $\bar \kappa >0$.
 We also know from Lemma~\ref{prop:equalA} that $I_t(A^\epsilon)$ equals the Young integral
$\int_0^t h(x_s, y_s^\eps) \, dB_s$.
Applying Lemma~\ref{sewing-lemma}, this leads to the bound
\begin{equs}
\Big\|\int_s^t h(x_r^\epsilon, y_r^\epsilon)\,dB_r\Big\|_p
&\lesssim  \bigl(\$A\$_{\bar \eta,p} |t-s|^{\bar \eta} + \|A\|_{\eta,p} |t-s|^{\eta}\bigr)\\
&\le C(h) \|x\|_{\alpha, p}\, \epsilon^{\bar \kappa}  \, |t-s|^{\bar \eta} 
+\epsilon^{\bar \kappa} |t-s|^{\eta}\;,
\end{equs}
uniformly over $\epsilon \in (0,1]$, thus completing the proof of (\ref{e:wantedUniformBound}).
\end{proof}

\begin{corollary}\label{co:inform-bounds}
Suppose that Assumption~\ref{ass:main} holds.
The solutions $x^\eps$ to \eqref{e:equation} are uniformly bounded in $\CB_{\alpha,p}$ for any $\alpha < H$ and $p\ge 1$.
\end{corollary}

\begin{proof}
The assumptions on our data guarantee that, for each $\eps > 0$, there exists a unique solution
to \eqref{e:equation} that belong to $\CB_{\alpha,p}$, so we only need to obtain the uniform bound.
By Theorem~\ref{fixed-point-map} we obtain for the time interval $[0,T]$ the bound
\begin{equ}
\|x\|_{\alpha,p} \lesssim  C(|f|_\infty, |f|_\Lip) T^\kappa \bigl(1 + \|x\|_{\alpha,p}\bigr)+T\,|g|_\infty \;,
\end{equ}
where $\kappa>0$, which implies the required bound on a sufficiently short time interval. One concludes by
iterating the bound. 
\end{proof}

\begin{remark}\label{re:g}
It is clear from the proof of Corollary~\ref{co:inform-bounds} that instead of assuming
that $g$ is bounded, it suffices to guarantee that it satisfies a bound of the form
$\|\int_0^T g(x_r, y_r^\epsilon)dr\|_p \lesssim T^\kappa (1 + \|x\|_{\alpha,p})$ for some
$\kappa > 0$. (Here $y_r^\eps$ solves \eqref{fast1} driven by $x$ as usual.)
\end{remark}

\subsection{Bounds on the random semigroup} 

In the rest of this section, we fix an $\CF_t$-adapted stochastic process $x_t$ and as usual $\Phi_{s,t}^{x}$ denotes the solution flow to (\ref{fast1}). For any $\bar x\in \R^d$, fixed,  $\Phi_{s,t}^{\bar x}$ denotes the solution to (\ref{e:xfrozen}) with frozen variable $\bar x$.
We first bound the difference between the evolutions of $\Phi^x$ and $\bar \Phi^{\bar x}$ over a short time period $[r,r']$.
\begin{lemma}\label{lemma:semi-group-2}
Suppose that $x\in \CB_{\alpha,p}$.
Let $F: \Omega\times \CY\to \R$ be bounded and $(\CF_s\vee \CG_u)\otimes \CB(\CY)$ measurable. Then,
for $s \le r < r' \le u$ with $|r'-r| \le \eps$ and for $\bar x \in \R^d$, one has the almost sure bound
\begin{equ}\label{semi-group-difference1}
|\hat \CP_{r,r'}^{\bar x}F - \CQ_{r,r'}^x F|_\infty \lesssim \sqrt{\sup_{v\in [r,r']}\E \big(|x_v-\bar x|^2 \,|\,\CF_r \vee \CG_u\big)}\,  |F|_\Lip\;.
\end{equ}
\end{lemma}
\begin{proof}
Since
$\bar \Phi^{\bar x}_{r,r'}(y_s)$ depends on the filtration of $B$ only through the value of $y_s$, it is measurable with respect to
$\CF_r \vee \hat \CG^r_{r'}$, c.f. \eref{e:defNoise}. Since furthermore $\hat \CG^r_{r'}$ is independent of $\CG_u$,
it follows that
$$
	\big(\hat \CP_{r,r'}^{\bar x} F(\omega, \fat)\big)(y)= \E (F(\omega,  \bar \Phi^{\bar x}_{r,r'}(y) )\,|\,\CF_r\vee \CG_u) \;.
$$
We now have
\begin{equs}
\bigl|\bigl(\CP_{r,r'}^{\bar x}F - \CQ_{r,r'}^x F\bigr)(\omega,y)\bigl| &=
\bigl|\E \bigl( F(\omega,  \bar \Phi^{\bar x}_{r,r'}(y) ) - F(\omega, \Phi^x_{r,r'}(y) )\,|\, \CF_r\vee \CG_u \bigr)\bigl| \\
&\le |F|_\Lip\; \E \bigl( \rho( \bar \Phi^{\bar x}_{r,r'}(y), \Phi^x_{r,r'}(y) )\,|\, \CF_r\vee \CG_u \bigr)\;.
\end{equs}
We then apply It\^o's formula to $d( \bar \Phi^{\bar x}_{r,r'}(y), \Phi^x_{r,r'}(y) )$, where $d$ is a
modification of $\rho$ such that $d^2$ is smooth. Since the increments of $\hat W$ on $[r, r']$ are independent of $\CF_r\vee \CG_u $,  its martingale term vanishes after taking conditional expectation with respect to $\CF_r\vee \CG_u$.
The rest of the estimate for the distance is routine, see Lemma~\ref{lemma:bounddiffy},
and the required bound follows.
\end{proof}

We fix a `final time' $u$  and recall that $\CU_s^u$ is the space of bounded real valued
functions from $\Omega\times \CY$ that are measurable with respect to $(\CF_s\vee \CG_u)\otimes \CB(\CY)$,  c.f. (\ref{CB}).

\begin{lemma}\label{lemma:Q}
 Let $s<r<v$.   Let $F\in \CU_s^u$ be a function that is continuous in the second variable for almost every $\omega$.
  The operators $\CQ_{s,r}^x: \CU^u_r\to \CU^u_s$ defined by (\ref{Q-semi-group}) 
satisfy the composition rule 
\begin{equ}
\CQ_{s,r}^x \circ \CQ_{r,v}^x F=\CQ_{s,v}^x F\;.
\end{equ}
\end{lemma}

\begin{proof}
Since $\Phi_{r,v}^{x}(\Phi_{s,r}^{x}(y,\omega),\omega) = \Phi_{s,v}^x(y,\omega)$
and  $\omega \mapsto \Phi_{s,r}^{x}(y,\omega)$ is $(\CF_r \vee \CG_u)$-measurable, for $F\in \CU_v^u$,
\begin{equs}
\bigl(\CQ_{s,r}^x \CQ_{r,v}^x F\bigr)(y)
&= \E \Bigl(\bigl(\CQ_{r,v}^x F\bigr)(\fat, \Phi_{s,r}^{x}(y)) \,\Big|\,\CF_s\vee \CG_u\Bigr) \\
&= \E \Bigl(\E \big( F(\fat, \Phi_{r,v}^{x}(\Phi_{s,r}^{x}(y)))\,\big|\,\CF_r\vee \CG_u \big) \,\Big|\,\CF_s\vee \CG_u\Bigr)\\
&= \E \big( F(\fat, \Phi_{s,v}^{x}(y))\,\big|\,\CF_s\vee \CG_u \big) = \bigl(\CQ_{s,v}^x F\bigr)(y)\;,
\end{equs}
as required. 
Here, the fact that $\omega \mapsto \Phi_{s,r}^{x}(y,\omega)$ is $(\CF_r \vee \CG_u)$-measurable
was used in order to go from the first to the second line. This is a particular instance of the fact that 
if $x \mapsto F(x,\omega)$ is continuous in $x$  for almost every $\omega$ 
 and $Y \colon \Omega \to \CX$ is a $\CG$-measurable random variable for some sub-$\sigma$-algebra $\CG$, then the identity
\begin{equ}
\E \big(F(x,\fat)\,|\, \CG\big)\big|_{x = Y} = \E \big(F(Y(\fat),\fat)\,|\, \CG\big)
\end{equ}
holds almost surely. We have also used the fact that $\Phi_{s,\fat}^x$ has a version which is continuous in time and in the initial value.
\end{proof}

\begin{lemma}\label{lemma:QP}
The following estimate holds uniformly over all $s < v \le u$
and all $F \in \CU_v^u$:
\begin{equ}\label{p-q}
|\hat \CP_{s,v}^{x_s}F - \CQ_{s,v}^x F|_\infty 
\lesssim  \sqrt{\E (|x|_\alpha^2 \,|\,\CF_s \vee \CG_u)}\, |v-s|^\alpha |F|_\Lip\;.
\end{equ}
\end{lemma}
\begin{proof}
We know that \eqref{p-q} holds for $|v-s| \le \eps$ from Lemma~\ref{lemma:semi-group-2} with
 $\bar x=x_s$ since one has the bound
 \begin{equ}[e:simpleBound]
\E \big(|x_{r'}- x_s|^2 \,|\,\CF_r \vee \CG_u\big)
\lesssim {{\E (|x|_\alpha^2 \,|\,\CF_s \vee \CG_u)}\, |u-s|^{2\alpha}}\;,
 \end{equ}
uniformly over $r \in [s,u]$ and $r' \in [r,u]$.
We also know from Lemma~\ref{lemma-semi-group-0} that $|\CP_{v-s}^{x_s}F|_\Lip$ is
bounded by a constant multiple of  $|F|_\Lip$, uniformly in time.

We then consider a partition $\Delta$ of $[s,v]$ into subintervals of size between $\eps/2$ and $\eps$,
and we write the difference of the two semigroups
as a telescopic sum, then apply consecutively the following: the triangle inequalities, 
the contraction property of $\CQ^x_{s,r}$, estimate (\ref{semi-group-difference1}) combined
with \eqref{e:simpleBound}, and Lemma~\ref{lemma-semi-group-0}. We also use the quasi semi-group property of $\CQ^x$.
This yields
\begin{equs}
|\hat \CP_{s,v}^{x_s}F - \CQ_{s,v}^x F|_\infty 
&\le \sum_{[r,r'] \in \Delta} |\CQ_{s,r}^x \bigl(\hat \CP_{r,r'}^{x_s} - \CQ_{r,r'}^x\bigr)\CP_{v-r'}^{x_s} F|_\infty \\
&\le \sum_{[r,r'] \in \Delta} |\bigl(\hat \CP_{r,r'}^{x_s} - \CQ_{r,r'}^x\bigr)\CP_{v-r'}^{x_s} F|_\infty\\
&\lesssim \sum_{[r,r'] \in \Delta} \sqrt{\E (|x|_\alpha^2\,|\, \CF_s \vee \CG_u)} \,|s-r'|^\alpha |\CP_{v-r'}^{x_s} F|_\Lip\\
&\lesssim \sum_{[r,r'] \in \Delta} \sqrt{\E (|x|_\alpha^2\,|\, \CF_s \vee \CG_u)} \,|s-v|^\alpha e^{-c|v-r'|/\eps}| F|_\Lip\\
&\lesssim  \sqrt{\E (|x|_\alpha^2 \,|\,\CF_s \vee \CG_u)}\, |v-s|^\alpha |F|_\Lip\;,
\end{equs}
as required. Note that there exists $C$ such that for any $\eps$,
$\sum_{[r,r'] \in \Delta} e^{-c|v-r'|/\eps}\le C$  provided the size of the partition is of order $\eps$.
\end{proof}

\subsection{Proof of the main result}
\label{sec:proofmain}

We now have all the ingredients in place for the proof of Theorem~\ref{thm:main}.

\begin{proof}[of Theorem~\ref{thm:main}]
By Lemma~\ref{lem:regularfbar}, we know that  $\bar f$ and $\bar g$ belong to $\BC^2$. 
This implies  that there exists a unique solution $\bar x_t$ to
\begin{equ}
d\bar x_t = \bar f(\bar x_t)\,dB_t + \bar g(\bar x_t)\,dt\;, \qquad \bar x_0=x_0,
\end{equ}
where the integral against $B$ is interpreted pathwise as a Young integral, see
for example \cite{TerryOld}.
We apply
Theorem~\ref{fixed-point-map} with $h = f-\bar f$, yielding the bound
\begin{equ}
\left\|\int_0^{\fat} \big( f(x_r, y_r^\eps) -\bar f(x_r) \big)\, dB_r\right\|_{\beta,p} \lesssim
  \eps^\kappa \bigl(1+\|x\|_{\alpha, p}\bigr)\;,
\end{equ}
uniformly over $x$ and over $\eps \in (0,1]$, where $y^\eps$ is obtained from $x$ by
solving \eqref{e:equationy}. Here, $\kappa> 0$ is small enough and $\alpha < {1\over 2}$ and
$p$ are such that $\alpha + H > 1 + \f1p$.
Since  $\sup_\epsilon\|x_s^\epsilon\|_{\alpha, p}<\infty$ by Corollary~\ref{co:inform-bounds}, we conclude that 
$$
\Big\|\int_0^{\fat} \left ( f(x_s^\epsilon, y_s^\epsilon) -\bar f(x_s^\epsilon) \right)dB_s\Big\|_{\beta, p}\lesssim \epsilon^\kappa\;,
$$
and a similar bound holds for
$\|\int_0^t \left ( g(x_s^\epsilon, y_s^\epsilon) -\bar g(x_s^\epsilon) \right)ds\|_{\beta, p}$. Setting
$$\bar x_t^\epsilon\eqdef x_0+\int_0^t \bar f(x_s^\epsilon) \, dB_s+ 
\int_0^t \bar g(x_s^\epsilon) \, ds\;,$$
we have just shown that the processes $x_t^\epsilon $ and  $\bar x_t^\epsilon$  are close in $\CB_{\beta, p}$:
\begin{equ}[e:boundAlmostDone]
\|\bar x^\epsilon-x^\epsilon\|_{\beta, p} \lesssim   \epsilon^\kappa.
\end{equ}
It remains to show that $\bar x_t^\epsilon$ and $\bar x_t$ are close in the $\beta$-H\"older norm, for which we begin 
by obtaining a pathwise estimates on their $\beta$-H\"older norm. 
Writing
$$ x_t^\epsilon =  x_t^\epsilon- \bar x_t^\epsilon+x_0+\int_0^t \bar f(x_s^\epsilon) \, dB_s+ 
\int_0^t \bar g(x_s^\epsilon) \, ds\;,$$
we may apply Lemma~\ref{general-convergence} to compare $x_t^\epsilon$ and $\bar x_t$,  where we take $F=(\bar f, \bar g)$ and $b_t=(B_t(\omega),t)$, $Z_0=x_0$ and $\bar Z_0=x_0+x_t^\epsilon- \bar x_t^\epsilon$, we have the pathwise estimate:
\begin{equs}
| x_t^\epsilon-\bar x_t|_{\beta} 
&\lesssim \exp\left(C | B|_\beta^{1/\beta} +C+  C |\bar x^\epsilon-x^\epsilon |_\beta^{1/\beta} \right) | \bar x^\epsilon-x^\epsilon |_{\beta}\;.
\end{equs}
(Note that the $\beta$-Hölder seminorm of the constant $x_0$ vanishes.)
Since (modulo changing $\beta$ slightly), we already know from \eqref{e:boundAlmostDone}
that $|\bar x^\epsilon-x^\epsilon |_{\beta} \to 0$ in probability at rate $\eps^\kappa$, this
concludes the proof.
\end{proof}

\appendix
\section{Estimates on conditioned fBm}

The purpose of this appendix is to provide a proof of Lemma~\ref{lemma-kernel}, as well as
to provide an explicit representation of $R$ used in Lemma~\ref{lemma-covariance}.
For this, we first derive a suitable representation for the mixed derivative of the
covariance function $R$ of $\tilde B$.

\begin{lemma}\label{lem:R}
Let $R$ be as above, $c_1=\cst 1, c_3 = \cst 1 \cst 3$, and $c_2=-c_3  \int_0^{\infty} u^{H-{1\over 2}}(1+u)^{H-{5\over 2}}\,du$.
 Then 
for $r < s$, one has the identity
\begin{equ}[e:exprR]
\d_{r,s}^2 R(r,s) = c_1 r^{H-{1\over 2}} s^{H-{3\over 2}} + c_2 (s-r)^{2H-2} + 
c_3 \int_r^\infty v^{H-{1\over 2}} (s-r+v)^{H-{5\over 2}}\,dv\;.
\end{equ} 
\end{lemma}

\begin{proof}
Recall that one has the identity
\begin{equ}[e:basicG]
G(t) = (2H-1) \hat R'(t) - (t+1) \hat R''(t)\;,
\end{equ}
with
\begin{equ}
\hat R(t)  = \int_0^1 (1-s)^{H-{1\over 2}}(1+t-s)^{H-{1\over 2}}\,ds\;.
\end{equ}
We have
\begin{equs}[e:F']
F'(t) &= \cst1\int_0^1 (1-s)^{H-{1\over 2}}(1+t-s)^{H-{3\over 2}}\,ds\\
&= \cst1 t^{2H-1} \int_0^{1/t} u^{H-{1\over 2}}(1+u)^{H-{3\over 2}}\,du
\end{equs}
where we used the change of variables $s=1+tu$.  Differentiating
the second line of \eqref{e:F'} immediately gives
\begin{equ}[e:F'']
t \hat R''(t) = (2H-1)\hat R'(t) - c_1 (1+t)^{H-{3\over 2}}\;.
\end{equ}
On the other hand,  differentiating
the first line of \eqref{e:F'}, we obtain
\begin{equs}
\hat R''(t) &= \cst1 \cst3 \int_0^1 (1-s)^{H-{1\over 2}}(1+t-s)^{H-{5\over 2}}\,ds\\
&= c_3  \int_0^1 u^{H-{1\over 2}}(t+u)^{H-{5\over 2}}\,du\\
&= c_3  \int_0^\infty u^{H-{1\over 2}}(t+u)^{H-{5\over 2}}\,du-c_3  \int_1^\infty u^{H-{1\over 2}}(t+u)^{H-{5\over 2}}\,du\\ 
&= c_3 t^{2H-2} \int_0^{\infty} u^{H-{1\over 2}}(1+u)^{H-{5\over 2}}\,du - c_3  \int_1^\infty u^{H-{1\over 2}}(t+u)^{H-{5\over 2}}\,du
\;.
\end{equs} 
Substituting \eqref{e:F''} into \eqref{e:basicG},
we can then rewrite $G$ as $G(t) = c_1 (1+t)^{H-{3\over 2}} - \hat R''(t)$, so that for
 $c_2=- c_3 \int_0^{\infty} u^{H-{1\over 2}}(1+u)^{H-{5\over 2}}\,du$,
\begin{equ}
G(t) = c_1 (1+t)^{H-{3\over 2}} + c_2 t^{2H-2} 
 + c_3 \int_{1}^{\infty} u^{H-{1\over 2}}(t+u)^{H-{5\over 2}}\,du\;,
\end{equ}
and the claim follows by substituting this into \eqref{e:relRG}.
\end{proof}

\begin{proof}[of Lemma~\ref{lemma-kernel}]
It follows from the
fact that a conditional variance is always smaller than the full variance that
\begin{equ}
|h|_{\RKHS}^2 \le C \Bigl| \int_0^T \int_0^T |r-s|^{2H-2}\, h(r)h(s)\,dr\,ds\Bigr| =: 2C|I|\;.
\end{equ}
By homogeneity, it suffices to consider the case $T=1$ and, by symmetry, we can restrict the domain
of integration to the region where $r \le s$.
To bound $I$, we then note that we can find a constant $c$ such that
\begin{equs}
I &=  \int_0^1  \int_0^s (s-r)^{2H-2} h(r)\,dr\, h(s)\,ds \\
&= c \int_0^1  \int_0^s \int_r^s (s-u)^{H-{3\over 2}}(u-r)^{H-{3\over 2}}\,du \, h(r)\,dr\, h(s)\,ds \\
&= c \int_0^1  \int_u^1 (s-u)^{H-{3\over 2}} h(s)\,ds \int_0^u (u-r)^{H-{3\over 2}}h(r)\,dr\,du\;.
\end{equs}
Performing one integration by parts, we note that 
\begin{equs}
\Big|\int_0^u (u-r)^{H-{3\over 2}}h(r)\,dr\Big| &= 
\Big|\cst3 \int_0^u (u-r)^{H-{5\over 2}}\bigl(\hat h(u) - \hat h(r)\bigr)\,dr\Big| \\
&\lesssim |h|_{-\kappa} u^{H-{1\over 2}-\kappa}\;,
\end{equs}
where we used the fact that $\kappa < H-{1\over 2}$ to guarantee integrability at $r = u$
and $\hat h$ denotes a primitive of $h$.
The other factor is bounded in the same way, so that
\begin{equ}
|I| \lesssim |h|_{-\kappa}^2\int_0^1 (1-u)^{H-{1\over 2}-\kappa}u^{H-{1\over 2}-\kappa}\,du\lesssim |h|_{-\kappa}^2\;,
\end{equ}
as required.
\end{proof}
%
%
%\begin{lemma}
%Assume that $F: \R^d \to \R$ has two bounded derivatives, 
%then $F: x \to (t\mapsto F(x_t))$ satisfies the bound
%\begin{equ}
%|F(x)-F(y)|_\alpha \lesssim |F'|_\infty |x-y|_\alpha +  |F''|_\infty |x-y|_\infty (|x|_\alpha + |y|_\alpha)\;.
%\end{equ}
%\end{lemma}
%\begin{proof}
%Given $x, y\in \CC_\alpha$, we aim to show that 
%\begin{equ}[e:wantedF]
% {|F(x_t)-F(y_t)-F(x_s)+F(y_s)|} \le C  |x-y|_\alpha |t-s|^\alpha \;.
%\end{equ}
%Taylor expanding the left hand side, we get the bound
%$$\big|F'(x_s) (x_t-x_s)-F'(y_s) (y_t-y_s)\big| + |F''|_\infty \big(|x_t-x_s|^2+|y_t-y_s|^2\big)\;.$$
%The last term is bounded by 
%\begin{equ}[e:bound1]
%|F''|_\infty (|x|_\alpha^2 + |y|_\alpha^2) |t-s|^{2\alpha}\;.
%\end{equ}
%The first term on the other hand is bounded by 
%\begin{equs}
%|F'(x_s) &(x_t-x_s- y_t+y_s)| + |F''|_\infty |x_s-y_s||y_t-y_s| \\
%&\le |F'|_\infty |x-y|_\alpha |t-s|^\alpha + |F''|_\infty |x-y|_\infty |y|_\alpha |t-s|^\alpha\;.
%\end{equs}
%On the other hand, we can Taylor expand the left hand side of \eqref{e:wantedF}
%in a slightly different way, yielding a bound of the type
%\begin{equ}
%	\big|F'(x_t) (x_t-y_t)-F'(x_s) (x_s-y_s)\big| + |F''|_\infty \big(|x_t-y_t|^2+|x_s-y_s|^2\big)\;.
%\end{equ}
%Similarly to before, the last term is bounded by 
%\begin{equ}[e:bound2]
%|F''|_\infty |x-y|_\infty^2\;,
%\end{equ}
%while the first term is bounded by 
%\begin{equ}
%	|F'|_\infty |x-y|_\alpha |t-s|^\alpha + |F''|_\infty |x-y|_\infty |x|_\alpha |t-s|^\alpha\;.
%\end{equ}
%Interpolating between \eqref{e:bound1} and \eqref{e:bound2} the claim follows.
%\end{proof}
%
%
\endappendix

\bibliographystyle{Martin}
\bibliography{./refs}

\end{document}